\documentclass[leqno]{amsart}
\usepackage{graphicx}
\newcommand{\nc}{\newcommand}
\newtheorem{theorem}{Theorem}[section]

\newtheorem{lemma}[theorem]{Lemma}
\newtheorem{proposition}[theorem]{Proposition}
\theoremstyle{definition}
\newtheorem{definition}[theorem]{Definition}
\theoremstyle{remark}

\theoremstyle{example}
\newtheorem{example}[theorem]{Example}
\numberwithin{equation}{section}
\nc{\ess}{\mathop{ess\;sup}}
\nc{\esi}{\mathop{ess\;inf}}
\begin{document}

\title[]{Operators  of Harmonic Analysis in Variable Exponent Lebesgue Spaces. Two--weight Estimates}%
\author{Vakhtang  Kokilashvili, Alexander  Meskhi and Muhammad Sarwar} %
%\address{}%
%\email{}%

%\thanks{}%
%\subjclass{}%
%

%\date{}%
%\dedicatory{}%
%\commby{}%
% ----------------------------------------------------------------
\maketitle
\noindent{\bf Abstract.} In the paper two--weighted norm estimates  with general weights for Hardy-type transforms, maximal functions, potentials and Calder\'on--Zygmund singular integrals in variable exponent Lebesgue spaces defined on quasimetric measure spaces $(X, d, \mu)$ are established. In particular,  we derive integral--type easily verifiable sufficient conditions governing two--weight inequalities for these operators. If exponents of Lebesgue spaces are constants, then most of the derived conditions are simultaneously necessary and sufficient for appropriate inequalities.  Examples of weights governing the boundedness of  maximal, potential and singular operators in weighted variable exponent Lebesgue spaces are given.

\vskip+0.2cm
\noindent{\bf Key words:} Variable exponent Lebesgue spaces, Hardy transforms, fractional and singular integrals, quasimetric measure spaces, spaces of homogeneous type, two-weight inequality.
\vskip+0.2cm


%\bibliographystyle{amsplain}
%\bibliography{}

\vskip+1cm
\section*{Introduction}
We study the two-weight problem  for Hardy-type,  maximal,
potentials and singular operators in Lebesgue spaces with
non-standard growth defined on quasimetric measure spaces. In
particular, our aim is to derive easily verifiable sufficient
conditions for the boundedness of these operators in weighted
$L^{p(\cdot)}(X)$ spaces which enable us effectively construct
examples of appropriate weights. The conditions  are
simultaneously necessary and sufficient for corresponding
inequalities  when the weights are of special type and the
exponent $p$ of the space is constant. We assume that the exponent
$p$ satisfies the local log-H\"older continuity  condition and if
the diameter of $X$ is infinite, then we suppose that $p$ is
constant outside some ball. In the framework of variable exponent
analysis such a condition first appeared  in the paper \cite{Di1},
where the author established the boundedness of the
Hardy--Littlewood maximal operator in $L^{p(\cdot)}({\Bbb{R}}^n)$.
As far as we know, unfortunately, it is not known an analog of the
log-H\"older decay condition (at infinity) for $p: X\to [1,
\infty)$  even in the unweighted case, which is well--known and
natural for the Euclidean spaces (see \cite{CrFiNe}, \cite{Ne},
\cite{CaCrFi}). The local  log-H\"older continuity   condition for
the exponent $p$ together with the log-H\"older decay condition
guarantees the boundedness of operators  of harmonic analysis in
$L^{p(\cdot)}({\Bbb{R}}^n)$ spaces (see e.g., \cite{CrFiMaPe}).

A considerable interest of researchers is attracted to the study
of mapping properties of integral operators defined on
(quasi-)metric measure spaces. Such spaces with doubling measure
and all their generalities naturally arise when studying boundary
value problems for partial differential equations with variable
coefficients, for instance, when the quasi–metric might be induced
by a differential operator, or tailored to fit kernels of integral
operators. The problem of the boundedness of integral operators
naturally arises also in the Lebesgue spaces with non-standard
growth.

Historically the boundedness of the Hardy--Littlewood maximal,
potential and singular operators in $L^{p(\cdot)}$ spaces defined
on (quasi)metric measure spaces was derived in \cite{HaHaLa},
\cite{HaHaPe}, \cite{KoMe3}, \cite{KoMe5},
\cite{KoSa3}-\cite{KoSa7}, \cite{AlSa} (see also references cited
therein).

Weighted inequalities for classical operators in $L^{p(\cdot)}_w$
spaces, where $w$ is a power--type weight, were established  in
the papers \cite{KoSa1}-\cite{KoSa3}, \cite{KoSa5}-\cite{KoSa7},
\cite{KoSaSa2}, \cite{EdMe}, \cite{SaVa}, \cite{SaSaVa},
\cite{DiSa} (see also the survey papers \cite{Sa3}, \cite{Ko}),
while the same problems with general weights for Hardy, maximal,
potential and singular operators were studied in
\cite{EdKoMeJFSA}-\cite{EdKoMe2}, \cite{KoMe4}, \cite{KoSa3},
\cite{KoSa5}, \cite{Kop}, \cite{Di3}, \cite{AsiKoMe}, \cite{MaZe},
\cite{DiHa}. Moreover, in the paper \cite{DiHa} a complete
solution of the one--weight problem for maximal functions defined
on Euclidean spaces are given in terms of Muckenhoupt--type
conditions. Finally we notice that in the paper \cite{EdKoMe2}
modular--type sufficient conditions  governing the two--weight
inequality for maximal and singular operators were established.

%while the same problems with general weights for Hardy, maximal, potential and singular operators   were studied in , \cite{KoMe4}, \cite{KoSa3}, %$[33-34]$, \cite{Di3}, \cite{AsiKoMe}, \cite{DiHa}. Moreover, in the paper \cite{DiHa} a complete solution of the one--weight problem for maximal %functions defined on Euclidean spaces are given in terms of Muckenhoupt--type conditions.

It should be emphasized that in the classical Lebesgue spaces the
two--weight problem for fractional integrals is already solved
(see \cite{KoKr}, \cite{KoMi}) but it is often useful  to
construct concrete examples of weights from transparent and easily
verifiable conditions. This problem for singular integrals still
remains open. However, some sufficient conditions governing
two--weight estimates for the Calder\'{o}n--Zygmund operators were
given in the papers \cite{EdKo}, \cite{CrMaPe} (see also the
monographs \cite{EdKoMe}, \cite{Vo} and references cited therein).

To derive two--weight estimates for maximal, singular and
potential operators  we use the appropriate inequalities for
Hardy--type transforms on $X$ (which are also derived in this
paper).

The paper is organized as follows: In Section 1 we give some
definitions and auxiliary results regarding quasimetric measure
spaces and the variable exponent Lebesgue spaces. Section 2 is
devoted to the sufficient conditions governing  two--weight
inequalities for Hardy--type defined on quasimetric measure
spaces, while in Section 3 we study the two--weight problem for
potentials defined on quasimetric measure spaces. In Section 4 we
discuss weighted estimates for maximal and singular integrals.

Finally we point out that constants (often different constants in
the same series of inequalities) will generally be denoted by $c$
or $C$. The symbol $f(x) \approx g(x)$ means that there are
positive constants $c_1$ and $c_2$ independent of $x$ such that
the inequality $ f(x) \leq c_1 g(x) \leq c_2 f(x)$ holds.
Throughout the paper by the symbol $p'(x)$ is denoted the function
$p(x)/ (p(x)-1)$.

\section{preliminaries}

Let $X:=(X, d, \mu)$ be a topological space with a complete
measure $\mu$ such that the space of compactly supported
continuous functions is dense in $L^1(X,\mu)$ and there exists a
non-negative real-valued function  (quasi-metric) $d$ on $X\times
X$ satisfying the conditions:

\vskip+0.1cm

\noindent(i) $d(x,y)=0$ if and only if  $x=y$;

%\noindent(ii) $d(x,y)> 0$ for all $x\neq y$, $x,y\in X$;

\noindent(ii) there exists a constant $a_1> 0$, such that
$d(x,y)\leq a_1(d(x,z)+d(z,y))$ for all $x,\,y,\,z\in X$;

\noindent(iii) there exists a constant $a_0> 0$, such that $d(x,y)
\leq a_0 d(y,x) $  for all $x,\,y,\in X$.

We assume that the balls $B(x,r):=\{y\in X:d(x,y)<r\}$ are
measurable and $0\leq\mu (B(x,r))<\infty$ for all $x \in X$ and
$r>0$; for every neighborhood $V$ of $x\in X,$ there exists $r>0,$
such that $B(x,r)\subset V.$ Throughout the paper we also suppose
that $\mu \{x\}=0$ and that
$$  B(x,R) \setminus B(x, r) \neq \emptyset \eqno{(1)}$$
for all  $x\in X$, positive $r$ and $R$ with  $0< r < R < L$,
where
$$L:= diam \; (X) = \;\sup\{ d(x,y): x,y\in X\}. $$

We call the triple $(X, d, \mu)$ a quasimetric measure space. If
$\mu$ satisfies the doubling condition $ \mu(B(x, 2r))\leq c \mu
(B(x,r))$, where the positive constant $c$ does not depend on
$x\in X$ and $r>0$, then $(X, d, \mu)$ is called a space of
homogeneous type $(\hbox{SHT})$. For the definition and some
properties of an $SHT$ see, e.g., \cite{CoWe}, \cite{StTo},
\cite{FoSt}.

A quasimetric measure space, where the doubling condition is not assumed and may fail, is  called a
non-homogeneous space.

Notice that the condition $L<\infty$ implies that $\mu(X)<\infty$ because every ball in $X$ has a finite measure.

We say that the measure $\mu$ is upper Ahlfors $Q$- regular if there is a positive constant $c_1$ such that
$ \mu B(x,r) \leq c_1 r^Q $ for for all $x\in X$ and $r>0$. Further, $\mu$ is lower Ahlfors $q-$ regular
if there is a positive constant $c_2$ such that
$ \mu B(x,r) \geq c_2 r^q  $ for all $x\in X$ and $r>0$.  It is easy to check that if $L<\infty$, then
$\mu$ is lower Ahlfors regular (see also, e.g., \cite{HaHaPe}).

For the boundedness of potential operators in weighted Lebesgue spaces with constant
exponents on non-homogeneous spaces we refer, for example, to the monograph \cite{EdKoMe1} (Ch. 6)
and references cited therein.

Let $p$ be a non--negative $\mu-$ measurable function on $X$. Suppose that $E$ is a $\mu-$ measurable
set in $X$ and $a$ is a constant satisfying the condition $1<a<\infty$. Throughout the paper we use the notation:
\begin{gather*}
 p_-(E):= \inf_{E} p; \;\; p_+(E):= \sup_{E} p; \;\;\  p_-:= p_-(X); \;\; p_+:= p_+(X);\\
 \overline{B}(x,r):= \{ y\in X: \; d(x,y)\leq r \},\;\;  kB(x,r):= B(x,kr);  B_{xy}:=B(x,d(x,y)); \\
 \overline{B}_{xy}:=\overline{B}(x,d(x,y));
g_{B}:=\frac{1}{\mu({B})}\int\limits_{B}|g(x)| d\mu(x).
\end{gather*}
where $1<p_-\leq p_+<\infty$.

\vskip+0.1cm

Assume now that $1\leq p_-\leq p_+< \infty$. The variable exponent Lebesgue space $L^{p(\cdot)}(X)$ (sometimes it is denoted by $L^{p(x)}(X)$) is the class of all $\mu$-measurable functions $f$ on $X$ for which $ S_p(f):= \int\limits_{X} |f(x)|^{p(x)} d\mu(x) < \infty. $
The norm in $L^{p(\cdot)}(X)$ is defined as follows:

$$ \|f\|_{L^{p(\cdot)}(X)} = \inf \{ \lambda>0: S_{p}(f/\lambda) \leq 1 \}. $$

It is known (see e.g. \cite{KoRa}, \cite{Sa1}, \cite{KoSa1}, \cite{HaHaPe}) that $L^{p(\cdot)}$
space is a Banach space. For other properties of
$L^{p(\cdot)}$ spaces  we refer to \cite{Sh}, \cite{KoRa}, \cite{Sa1}, \cite{Sa3}, \cite{Ko}, etc.

\vskip+0.1cm

Now we introduce several definitions:
\vskip+0.1cm

%Definition 1.1.
\begin{definition} Let $(X, d, \mu)$ be a quasimetric measure space and let $N \geq 1$ be  a constant.
Suppose that $p$ satisfy
the condition $0<p_-\leq p_+<\infty$. We say that $p \in {\mathcal{P}}(N,x)$, where $x\in X$, if there are
positive constants
$b$ and $c$ (which might be depended on $x$) such that
$$ \mu(B(x, Nr))^{p_-(B(x,r))- p_+(B(x,r))} \leq c \eqno{(2)} $$
holds for all $r$, $0<r\leq b.$ Further, $p \in {\mathcal{P}}(N)$ if there are a positive constants $b$ and $c$ such
that (2) holds for all $x\in X$ and all $r$ satisfying the condition $0<r\leq b.$
\end{definition}

%Definition 1.2
\begin{definition}

Let $(X, d, \mu)$ be an $SHT$. Suppose that $0<p_-\leq
p_+<\infty$. We say that  $p\in LH(X,x)$ ( $p$ satisfies the
log-H\"older-- type condition at  a point $x\in X$) if there are
positive constants $b$ and $c$ (which might  be depended on $x$)
such that
$$  |p(x)-p(y)| \leq \frac{c}{-\ln\big(\mu (B_{xy}\big)\big)} \eqno{(3)}$$
holds for all $y$ satisfying the condition  $d(x,y)\leq b.$
Further,  $p\in LH(X)$ ( $p$ satisfies the log-H\"older type
condition on $X$)if there are positive constants $b$ and $c$ such
that $(3)$ holds for all $x,y$ with $d(x,y)\leq b.$
\end{definition}

%Definition 1.3.
\begin{definition}Let $(X, d, \mu)$ be a quasimetric measure space and let $0<p_-\leq p_+<\infty$.
We say that $p\in \overline{LH}(X,x)$ if there are  positive
constants $b$ and  $c$ (which might be depended on $x$) such that
$$  |p(x)-p(y)| \leq \frac{c}{-\ln d(x,y)}\eqno{(4)}$$
for all $y$ with $d(x,y)\leq b.$  Further, $p\in \overline{LH}(X)$
if (4) holds for all $x,y$ with $d(x,y)\leq b.$
\end{definition}

It is easy  to see that if the measure $\mu$ is upper Ahlfors
$Q$-regular and $p\in LH(X)$ (resp. $p\in LH(X,x)$), then $p\in
\overline{LH}(X)$ (resp. $p\in \overline{LH}(X,x)$. Further, if
$\mu$ is lower Ahlfors $q$-regular and $p\in \overline{LH}(X)$
(resp. $p\in \overline{LH}(X,x)$), then $p\in LH(X)$ (resp. $p\in
LH(X,x)$). \vskip+0.1cm

%Remark 1.1.
{\em Remark} 1.1. It can be checked easily that if $(X, d, \mu)$ is an SHT,
then $\mu B_{x_0x}\approx \mu B_{xx_0}.$

\vskip+0.1cm

{\em Remark} 1.2.  Let $(X, d, \mu)$ be an $\hbox{SHT}$ with
$L<\infty$. It is known (see,  e.g., \cite{HaHaPe}, \cite{KoMe3})
that if $p\in{\overline{LH}}(X)$), then $p\in\mathcal{P}(1)$.
Further, if $\mu$ is upper Ahlfors $Q$-regular, then the condition
$p\in\mathcal{P}(1)$ implies that $p\in \overline{LH}(X)$.

%{\em Remark} 1.3. Due to the doubling condition for $\mu$ we have the following inequality
%$$ (\mu B_{xy})^{p(x)} \leq  c (\mu B(x, d(y,x)))^{p(y)}, $$
%where the positive constant $c$ does not depend on $x,y\in X$.

%Proposition 1.4.
\begin{proposition} If $0<p_-(X)\leq p_+(X)<\infty$ and $p\in
LH(X)$ $($ resp. $p\in \overline{LH}(X) \;)$, then the functions
$c p(\cdot)$ and $1/p(\cdot)$ belong to the class $LH(X)$ $($
resp. $\overline{LH}(X)\;).$ Further if $p\in LH(X,x)$ (resp.
$p\in \overline{LH}(X,x)$)  then $cp(\cdot) and 1/p(\cdot)$ belong
to $LH(X,x)$ $($ resp. $p\in \overline{LH}(X,x)\; )$, where $c$ is
a positive constant. If $1<p_-(X)\leq p_+(X)<\infty$ and $p\in
LH(X)$ $($ resp. $p\in LH(X) \overline{LH}(X)\;)$, then $p'\in
LH(X,x)$ $($ resp. $p'\in \overline{LH}(X,x))$.
\end{proposition}

Proof of this statement   follows immediately from the definitions
of the classes $LH(X,x)$, $LH(X)$, $\overline{LH}(X,x)$,
$\overline{LH}(X)$; therefore we omit the details.

%Proposition 1.5.
\begin{proposition}
Let $(X, d, \mu)$ be an $\hbox{SHT}$ and let $p\in {\mathcal{P}}(1)$. Then $ (\mu B_{xy})^{p(x)}
\leq  c (\mu B_{yx})^{p(y)}, $
for all $x,y\in X$ with $\mu(B(x, d(x,y))) \leq b,$ where $b$ is a small  constant, and the
constant $c$ does not depend on $x,y\in X$.
\end{proposition}

\begin{proof} Due to the doubling condition for $\mu$, Remark $1.1$, the condition $p\in {\mathcal{P}}(1)$ and the
fact $x\in B(y, a_1(a_0+1)d(y,x))$ we have the following estimates: $ \mu (B_{xy})^{p(x)} \leq \mu \big( B(y, a_1(a_0+1)d(x,y))\big)^{p(x)} \leq c \mu B(y, a_1(a_0+1)d(x,y))^{p(y)} \leq c (\mu B_{yx})^{p(y)} $, which proves the statement.
\end{proof}

%{\bf Lemma A (\cite{HaHaPe}, \cite{KoMe3})} {\em Let $p_+<\infty$. If $\mu$ is lower Ahlfors regular, then the condition $p\in LH(X)$ implies $p\in %{\Bbb{P}}(1)$. If  $\mu$ is upper Ahlfors regular, then from the he condition $p\in {\Bbb{P}}(1)$ it follows that $p\in LH(X)$.}

The proof of the next statement is trivial and follows directly from the definition of the classes ${\mathcal{P}}(N,x)$ and ${\mathcal{P}}(N)$. Details are omitted.

%Proposition 1.6.
\begin{proposition} Let $(X, d, \mu)$ be a quasimetric measure space and let $x_0\in X$. Suppose
that $N\geq1$ be a constant.  Then the following statements hold:

\rm{(i)} If $p\in {\mathcal{P}}(N,x_0)$ (resp.  $p\in {\mathcal{P}}(N))$,  then there are positive constants
$r_0$, $c_1$ and $c_2$ such that
for all $0<r\leq r_0$ and all $y\in B(x_0,r)$ (resp. for all $x_0, y$ with $d(x_0,y) < r \leq r_0$) we have that
$\mu \big( B(x_0,Nr)\big)^{p(x_0)}\leq c_1 \mu \big( B(x_0,Nr)\big)^{p(y)}\leq c_2 \mu \big( B(x_0,Nr)\big)^{p(x_0)}.$

 \rm{(ii)} Let  $p\in {\mathcal{P}}(N,x_0)$. Then there are positive constants   $r_0$, $c_1$
and $c_2$ (in general, depending on $x_0$) such that for all $r$ ($r\leq r_0$) and all $x,y\in B(x_0,r)$ we have
$\mu\big( B(x_0,Nr)\big)^{p(x)}\leq c_1\mu\big( B(x_0,Nr)\big)^{p(y)}\leq c_2\mu\big( B(x_0,Nr)\big)^{p(x)}.$

 \rm{(iii)} Let  $p\in {\mathcal{P}}(N)$. Then there are positive constants   $r_0$, $c_1$
and $c_2$ such that for all balls $B$ with radius $r$ ($r\leq r_0$) and all $x,y\in B$, we have
$\mu(N B)^{p(x)}\leq c_1\mu(N B)^{p(y)}\leq c_2\mu( N B)^{p(x)}.$
\end{proposition}

It is known that (see, e.g., \cite{KoRa}, \cite{Sa1}) if $f$ is a measurable function on $X$ and  $E$ is a measurable subset of $X$, then
the following  inequalities hold:
\begin{gather*}
 \|f\|^{p_+(E)}_{L^{p(\cdot)}(E)} \leq S_{p} (f\chi_E) \leq
\|f\|^{p_-(E)}_{L^{p(\cdot)}(E)}, \;\; \|f\|_{L^{p(\cdot)}(E)}\leq
1; \\
 \|f\|^{p_-(E)}_{L^{p(\cdot)}(E)}\leq S_{p}(f\chi_E)\leq
\|f\|^{p_+(E)}_{L^{p(\cdot)}(E)},\;\; \|f\|_{L^{p(\cdot)}(E)}> 1.
\end{gather*}

H\"older's inequality in variable exponent Lebesgue spaces has the following form:

$$\int\limits_{E} fg d\mu \leq \Big(1/p_-(E) + 1/(p')_-(E) \Big) \| f \|_{L^{p(\cdot)}(E)}  \| g \|_{L^{p'(\cdot)}(E)}.$$
\vskip+0.1cm

%The next lemma for the Euclidean spaces was proved  in \cite{Kam} but we give the proof for completeness.
%\vskip+0.1cm

%{\bf Lemma 1.2.} {\em Let $r$ be a constant and $p(\cdot)$ be a measurable function on $X$ such that $1<p(x)\leq r<\infty$. Then there exists a %positive constant $c$ such that for all families of disjoint sets $\{B_k\}$ such that $\cup_kB_k= X$ the inequality
%$$ \sum_{k} \| f \chi_{B_k} \|^r_{L^{p(\cdot)}(X)} \leq c \|f\|_{L^{p(\cdot)}(X)}^r $$
%holds.}

%\vskip+0.1cm
%{\em Proof.} Let $a_k:=\| f \chi_{B_k}\|_{L^{p(\cdot)}(X)}$. Suppose that $\|f\|_{L^{p(\cdot)}(X)} \leq 1$. Then we have
%$$ \sum_{k} \| f \|_{L^{p(\cdot)}(X)}^r = \sum_{\{ k: a_k \leq 1\}}\| f \|_{L^{p(\cdot)}(X)}^r + \sum_{\{ k: a_k > 1\}}\| f \|_{L^{p(\cdot)}(X)}^r := %S_1+ S_2. $$

%Using Lemma 1.1 and the condition $r/p_+\geq 1$ we have
%$$ I_1 \leq \sum_{k} \bigg( \int_{B_k} |f(x)|^{p(x)} d\mu(x) \bigg)^{r/p_+} \leq  \bigg( \int_{X} \bigg( \sum_{k} |f(x)|^{p(x)} \chi_{B_k}(x) \bigg) %d\mu(x) \bigg)^{r/p_+}  \leq \| f\|_{L^{p(\cdot)}}^{rp_-/p_+} \leq 1. $$

%Taking  the fact $r/p_- \geq 1$ into account we obtain
%$$ I_{2} =  \sum_{\{ k: a_k > 1\}} \bigg( \int_{B_k} |f(x)|^{p(x)} d\mu(x) \bigg)^{r/p_-} \leq \bigg( \int_{X} |f(x)|^{p(x)} d\mu(x) \bigg)^{r/p_-} %\leq \| f \|_{L^{p(\cdot)}}^{rp_+/p_-} \leq 1. $$
% $\;\;\;\; \Box$
%\vskip+0.1cm

% The next lemma was proved in \cite{EdKoMe}, Chapters 5 and 6 (see also \cite{KoMe3}):

%Lemma 1.7
\begin{lemma}  Let $(X, d, \mu)$ be an $\hbox{SHT}$.

\rm{(i)} Let   $\beta$ be a measurable function on $X$ such that $\beta_+ < -1$ and let  $r$ be a small
positive number. Then there exists a positive constant $c$ independent of $r$ and $x$ such that
$$ \int\limits_{X \setminus B(x_0,r) } (\mu B_{x_{0}y})^{\beta(x)} d \mu(y) \leq c \frac{\beta(x)+1}{\beta(x)}
\mu(B(x_0,r))^{\beta(x)+1}; $$

\rm{(ii)}  Suppose that $p$ and $\alpha$ are measurable functions on $X$ satisfying the conditions
$1<p_- \leq p_+ <\infty$ and $\alpha_- > 1/p_-$. Then there exists a positive constant $c$ such that for
all $x\in X$ the inequality

$$ \int\limits_{\overline{B} (x_0 , 2d(x_0 ,x))} \big( \mu B(x,d(x,y))\big)
^{(\alpha(x) -1)p'(x)} d\mu (y) \leq c \big( \mu B(x_0 ,d(x_0 ,x)) \big) ^{(\alpha(x) -1)p'(x) +1} $$
holds.
\end{lemma}
\begin{proof} Part (i) was proved in \cite{KoMe3} (see also \cite{EdKoMe}, p.372,
for constant $\beta$). The proof of Part (ii) was given in
\cite{EdKoMe} (Lemma 6.5.2, p. 348) but repeating those arguments
we can see that it is also true for variable $\alpha$ and $p$.
Details are omitted.
\end{proof}

%The boundedness of the Hardy-Littlewood, potential and singular operators in $L^{p(x)}(\Omega)$ $(\Omega \subseteq {\Bbb{R}}^n$) spaces was %established in \cite{Sa2}, \cite{Di1}, \cite{Di2}, \cite{DiRu}, \cite{CrFiNe}, \cite{CrFiMaPe}, \cite{KoSa2}, \cite{KoSa3}, \cite{EdKoMe3}, %\cite{KoMeMorreyNew}, \cite{SaFCAANew}, etc.. The same problem for metric measure spaces was studied in \cite{EdKoMe1}, \cite{HaHaPe}, \cite{HHL} %(New), \cite{FHH...}(new), \cite{KoSa3}, \cite{KoSa4},  \cite{KoMeArmJM} etc..  For weighted estimates for these operators in variable exponent %Lebesgue spaces we refer to \cite{KoSa1}-\cite{KoSa5}, \cite{EdKoMeJFSA}, {\cite{EdMeMZ}, \cite{EdKoMeHJM}, \cite{KoMeITSF},\cite{Di3}, \cite{MubSa}, %\cite{KoMeProc}, \cite{KoMeProc2}, \cite{AshKoMeMIA}, etc.

Let $M$ be a maximal operator on $X$ given by
$$ Mf(x):= \sup_{x\in X,r>0} \frac{1}{\mu(B(x,r))}\int\limits_{B(x,r)}|f(y)| d\mu(y). $$

%Definition 1.8
\begin{definition} Let $(X, d, \mu)$ be a quasimetric measure space. We say that $p\in {\mathcal{M}}(X)$ if the operator $M$ is bounded in $L^{p(\cdot)}(X)$.
\end{definition}

%{\bf Definition 1.4.} Let $1<p_-\leq p_+<\infty$. We say that $p\in {\Bbb{M}}(X)$ if the maximal operator $M$ is bounded in $L^{p(\cdot)}(X)$.
%\vskip+0.1cm

L. Diening \cite{Di1} proved that if $\Omega$ is a bounded domain
in ${\Bbb{R}}^n$, $1< p_-\leq p_+<\infty$ and $p$ satisfies the
local log-H\"older continuity  condition on $\Omega$ (i.e., $
|p(x)-p(y)|\leq\frac{c}{-\ln (|x-y|)}$ for all $x,y\in \Omega$
with $|x-y|\leq 1/2$), then the Hardy--Littlewood maximal operator
defined on $\Omega$ is bounded in $L^{p(\cdot)}(\Omega)$.
\vskip+0.1cm

%{\bf Theorem A (\cite{HaHaPe}, \cite{KoMe3})}: {\em Let $(X, d,\mu)$ be an \hbox{SHT} with $L<\infty$. If $1<p_-\leq p_+<\infty$ and $p\in \overline{LH}(X)$, then $p\in {\mathcal{M}}(X)$.}
%\vskip+0.1cm

%\vskip+0.1cm
%Since extrapolation theorem for variable exponent Lebesgue spaces is true also for quasimetric measure spaces (see \cite{CRFiMaPe}, \cite {KoSa6}), we %have the next two statements:
%\vskip+0.1cm

%We shall also need the following regarding non-homogeneous spaces.

%\vskip+0.1cm

%For one-weight inequalities it is known the following statement (see \cite{KoSa3}, \cite{KoSa5}).
%\vskip+0.1cm

%{\bf Theorem B.} {\em Let $(X, d, \mu)$ be  a metric measure space with doubling condition and let  $L<\infty$. Suppose that $1<p_-\leq p_+<\infty$ %and $p\in {\overline{LH}}(X)$. If the weight function $\rho$ on $X$ satisfies the condition
%$$ \sup_{B} \bigg( \frac{1} {\mu(B)} \int_B \rho^{p(x)}d\mu(x)\bigg)^{1/p_-} \bigg( \frac{1}{\mu(B)} \int_B \rho^{p(x)(1-(p_-)'} %d\mu(x)\bigg)^{1/(p_-)'}<\infty,$$
%then there is a positive constant $C$ such that
%$$ \| \rho Mf \|_{L^{p(\cdot)}(X)} \leq C \| \rho f \|_{L^{p(\cdot)}(X)} $$
%holds.}

%\vskip+0.1cm
%{\bf Theorem F.}

Now we prove the following lemma:
%Lemma 1.9
\begin{lemma} Let $(X, d, \mu)$ be an $SHT$.   Suppose that $0<p_-\leq p_+<\infty$.
Then $p$ satisfies the condition $p\in {\mathcal{P}}(1)$
$($resp.  $p\in {\mathcal{P}}(1,x)$) if and only if  $p\in LH(X)$
$($ resp. $p\in LH(X,x)\;)$.
\end{lemma}

\begin{proof} {\em Necessity.}  Let $p\in {\mathcal{P}}(1)$  and let  $x, y\in X$ with $d(x,y)< c_0$ for some positive constant $c_0$. Observe that $x,y\in B$, where $B:=B(x, 2d(x,y))$. By the doubling condition for $\mu$ we
have that $ \big(\mu B_{xy}\big)^{-|p(x)- p(y)|} \leq c \big( \mu
B \big)^{-|p(x)- p(y)|}  \leq c \big( \mu B \big)^{p_-(B) -
p_+(B)} \leq C, $ where $C$ is a positive constant which is
greater than $1$. Taking now the logarithm in the last inequality
we have that $p\in LH(X)$. If $p\in {\mathcal{P}}(1,x)$, then by
the same arguments we find that  $p\in LH(X,x)$.

{\em Sufficiency.} Let $B:=B(x_0,r)$. First observe that If $x,y
\in B$, then  $\mu B_{xy} \leq c \mu B(x_0,r)$. Consequently, this
inequality and the condition $p\in LH(X)$ yield $ |p_-(B)- p_+
(B)| \leq \frac{C}{- \ln \big(c_0 \mu B(x_0, r)\big)}$. Further,
there exists $r_0$ such that $0<r_0<1/2$ and $ c_1 \leq \frac{ \ln
\big(\mu (B)\big)}{-\ln \big(c_0 \ln \mu(B)\big)} \leq c_2$, $0< r
\leq r_0 $, where $c_1$ and $c_2$ are positive constants. Hence $
\big(\mu(B)\big)^{p_-(B)- p_+(B)}  \leq \Big( \mu(B) \Big)^{
\frac{C}{\ln \big( c_0\mu (B)\big)} } = \exp\bigg( \frac{C \ln
\big( \mu(B)\big) }{\ln \big(c_0 \mu(B)\big)}\bigg) \leq C$.

Let now $p\in LH(X,x)$ and let $B_x:=B(x,r)$  where $r$  is a
small number. We have that $ p_+(B_x)- p(x) \leq \frac{c}{- \ln
\big(c_0 \mu B(x, r)\big)}$ and $ p(x)-p_-(B_x) \leq \frac{c}{-
\ln \big(c_0 \mu B(x, r)\big)}$ for some positive constant $c_0$.
Consequently, $ (\mu (B_x) )^{p_-(B_x)- p_+(B_x)}  = \big(
\mu(B_x)  \big)^{p(x)- p_+(B_x)} \big( \mu(B_x) \big)^{p_-(B_x)-
p(x)}\leq c \big( \mu(B_x) \big)^{\frac{-2c}{-\ln(c_0\mu
B_x))}}\leq C$.
\end{proof}

\begin{definition} A measure $\mu$  on $X$ is said to satisfy the reverse   doubling condition $( \mu \in \hbox{RDC}(X))$ if there exist  constants $A>1$ and $B>1$  such that the inequality $    \mu\big(B(a,Ar)\big)\geq  B\mu\big(B(a,r) \big)$ holds.
\end{definition}
{\em Remark} 1.3. It is known that if all annulus in $X$ are not empty, then $\mu \in \hbox{DC}(X)$ implies that $\mu \in \hbox{RDC}(X)$ (see, e.g., \cite{StTo}).
\vskip+0.1cm

In the sequel we will use the notation:

\begin{gather*}
 I_{1,k} := \begin{cases} B(x_0,A^{k-1}L / a_1)\; \hbox{if} \;\; L<\infty  \\
B(x_0,A^{k-1}/ a_1)\;\; \hbox{if}\;\; L = \infty \end{cases},   \\
 I_{2,k} := \begin{cases} \overline{B}(x_0,A^{k+2}a_1L)\setminus B(x_0, A^{k-1}L/ a_1),\;\; \hbox{if} \;\;L<\infty \\
\overline{B}(x_0,A^{k+2}a_1 )\setminus B(x_0, A^{k-1}/ a_1)\;
\hbox{if}\;\; L=\infty, \end{cases},\\
 I_{3,k} := \begin{cases} X\setminus B(x_0,A^{k+2}L  a_1)\;\; \hbox{if} \;\; L<\infty  \\
 X\setminus B(x_0,A^{k+2} a_1)\;\; \hbox{if}\; L = \infty \end{cases},
 \\
 E_k := \begin{cases} \overline{B}(x_0,A^{k+1}L) \setminus B(x_0,A^{k}L) \;\; \hbox{if}\;\; L<\infty \\
\overline{B}(x_0,A^{k+1}) \setminus B(x_0,A^{k}) \;\; \hbox{if}\;\; L=\infty \end{cases},
\end{gather*}
where the constant $A$ is defined in the reverse doubling condition and the constant $a_1$ is taken from the triangle inequality for the quasimetric $d$.

%Lemma 1.11
\begin{lemma} Let $(X, d, \mu)$ be an $\hbox{SHT}$.  Suppose
that there is a point $x_0\in X$ such that $p\in  LH(X,x_0)$. Then
there exist   positive constants  $r_0$ and $C$ $($ which might be
depended on $x_0\;)$ such that for all $r$, $0<r\leq r_0$, the
inequity
$$( \mu B_{A})^{p_{-}(B_{A})-p_{+}(B_{A})}\leq C$$
holds,  where $B_{A}: = B(x_0, Ar) \setminus B (x_0, r)$ and
the constant $C$ is independent of $r$ and the constant $A$ is defined in Definition 1.10.
\end{lemma}

\begin{proof} Let $B:= B(x_0,r)$. First observe that by the doubling and reverse doubling conditions we have that
$ \mu B_{A} = \mu B(x_0,Ar)-\mu B(x_0,r) \geq (B-1)\mu B(x_0,r) \geq c \mu (AB)$.  Suppose that $0<r<c_{0}$, where $c_0$ is a sufficiently small constant. Then by using Lemma 1.9 we find that $   \big(\mu B_{A}\big)^{p_{-}(B_{A})-p_{+}(B_{A})} \leq c \big( \mu( AB ) \big)^{p_{-}(B_{A})-p_{+}(B_{A})}\leq  c\big( \mu( AB)\big)^{p_{-}(AB)-p_{+}(AB)} \leq c. $
\end{proof}

%{\bf Proposition 1.3 (see, e.g., \cite{KoRa}, \cite{Sa1}).} {\em Let $(X, d, \mu)$ be a quasimetric finite measure space and let $1\leq r(x)\leq %p(x)$. Then there exist a positive constant $c$ depending on $r$ and $p$ such that
%$$    \|f\|_{L^{r(\cdot)}(X)}\leq (1+\mu(X)) \|f\|_{L^{p(\cdot)}(X)} $$}
%\vskip+0.1cm
%The next statement for Euclidean spaces was proved in \cite{Di1} (Lemma 2.2). For an $SHT$ follows in a similar manner; therefore we omit the proof.
%\vskip+0.1cm
%{\bf Proposition 1.4.} {\em Let $(X,d,\mu)$ be a quasimetric measure space and let $p$ and $q$ be bounded exponents on $\Omega$. Then
%$$ L^{q(\cdot)}(\Omega) \hookrightarrow L^{p(\cdot)}(\Omega)$$
%if $p(x)\leq q(x)$ almost everywhere on $X$ and there is a constant $0<K<1$ such that
%$$ \int_{X} K^{p(x)q(x)/ (|q(x)-p(x)|)} d\mu(x) <\infty. \eqno{(11)}$$.}
%\vskip+0.1cm

%Lemma 1.12
\begin{lemma} Let $(X,d,\mu)$ be an ${\hbox{SHT}}$ and let $
1<p_{-}(x)\leq p(x)\leq q(x)\leq q_{+}(X)<\infty.$  Suppose that
there is a point $x_0\in X$ such that $p,q\in LH(X,x_0)$.  Assume
that $p(x)\equiv p_{c}\equiv const $, $q(x)\equiv q_{c}\equiv
\hbox{const}$ outside some ball $B(x_{0},a)$ if $L=\infty$. Then
there exist a positive constant C such that
$$     \sum\limits_k\|f\chi_{I_{2,k}}\|_{L^{p(\cdot)}(X)}\|g \chi_{I_{2,k}}\|_{L^{q'(\cdot)}(X)}\leq C\|f\|_{L^{p(\cdot)}(X)}\|g\|_{L^{q'(\cdot)}(X)}
$$
for all $f\in L^{p(\cdot)}(X)$ and $g\in L^{q'(\cdot)}(X).$
\end{lemma}

\begin{proof} Suppose that $L =\infty$. To prove the lemma  first observe that
$  \mu(E_k)\approx \mu B(x_0,A^k )$ and  $\mu(I_{2,k})\approx\mu B(x_0,A^{k-1})$. This holds  because $\mu$ satisfies the reverse doubling condition and, consequently,
\begin{gather*}
 \mu E_k =  \mu \Big(\overline{B}(x_0,A^{k+1})\setminus B(x_0,A^k )\Big) =\mu \overline{B}(x_0,A^{k+1}) - \mu B{(x_0,A^{k})} \\
= \mu \overline{B}(x_0,A A^{k}) - \mu B(x_0,A^{k}) \geq B\mu B(x_0, A^{k}) - \mu B(x_0,A^{k}) = (B-1)\mu B(x_0, A^{k})
\end{gather*}
Moreover,  using the doubling condition we have $ \mu E_{k}\leq \mu B(x_0, AA^{k})\mathop \leq c \mu B(x_0, A^{k}), $
where $c>1$. Hence, $ \mu E_k\approx \mu B(x_0,A^{k}).$

%$$  (B-1)\mu B(x_0,A^{k})\leq \mu E_k\leq c \mu B(x_0,A^{k}) $$
%which implies

Further,  since we can assume that $a_1\geq 1$, we find that
 \begin{gather*}
 \mu I_{2,k} = \mu \Big(\overline{B}(x_{0},A^{k+2}a_1)\setminus B(x_{0},A^{k-1}/a_1)\Big) =\mu \overline{B}(x_{0},A^{k+2}a_1)-\mu B(x_{0},A^{k-1}/a_1) \\
 = \mu \overline{B}(x_{0},A A^{k+1} a_1)- \mu B(x_{0},A^{k-1}/a_1) \geq  B \mu B(x_{0}, A^{k+1} a_1)-\mu B(x_{0},A^{k-1}/a_1) \\
 \geq  B^{2}\mu B(x_{0},A^{k}/a_1)-\mu B(x_{0},A^{k-1}/a_1)  \geq B^{3}\mu B(x_{0},A^{k-1}/a_1)-\mu B(x_{0},A^{k-1}/a_1) \\
=(B^{3}-1)\mu B(x_0,A^{k-1}/a_1).
 \end{gather*}
Moreover, using the doubling condition we have $ \mu I_{2,k} \leq \mu\overline{ B}(x_{0},A^{k+2}r) \leq c \mu B(x_{0},A^{k+1}r) \leq  c^{2}\mu B(x_{0},A^{k}/a_1) \leq  c^{3}\mu B(x_{0},A^{k-1}/a_1) $. This gives the estimates $ (B^{3}-1)\mu B(x_{0},A^{k-1}/a_1) \leq  \mu (I_{2,k}) \leq c^{3}\mu B(x_{0},A^{k-1}/a_1).$

For simplicity assume that $a=1$. Suppose that $m_0$ is an integer such that $\frac{A^{m_0-1}}{a_1}>1$. Let us split the sum as follows:

$$ \sum\limits_{i}\|f \chi_{I_{2,i}}\|_{L^{p(\cdot)}(X)}\cdot \|g \chi_{I_{2,i}}\|_{L^{q'(\cdot)}(X)}=\sum\limits_{i\leq m_0}\Big(\cdots \Big)+\sum\limits_{i> m_0}\Big(\cdots\Big)=:J_{1}+J_{2}. $$

Since $p(x)\equiv p_{c}=const,\ q(x)=q_{c}=const$ outside the  ball $B(x_0,1)$, by using H\"{o}lder's inequality and the fact that $p_{c}\leq q_{c}$, we have
$$ J_{2} = \sum\limits_{i>m_0}\|f \chi_{I_{2,i}}\|_{L^{p_c}(X)}\cdot \|g \chi_{I_{2,i}}\|_{L^{(q_c)'}(X)}
        \leq c\|f\|_{L^{p(\cdot)}(X)}\cdot \|g\|_{L^{q'(\cdot)}(X)}. $$

Let us estimate $J_{1}$. Suppose that $
\|f\|_{L^{p(\cdot)}(X)}\leq 1$ and $\|g \|_{L^{q'(\cdot)}(X)} \leq
1$. Also, by Proposition 1.4 we have that $1/q' \in LH(X,x_0)$.
Therefore by Lemma 1.11 and the fact that $1/q'\in LH(X,x_0)$ we
obtain that $\mu \big( I_{2,k}\big)^{\frac{1}{q_{+}(I_{2,k})}}
\approx \|\chi_{I_{2,k}}\|_{L^{q(\cdot)}(X)}\approx \mu
\big(I_{2,k}\big)^{\frac{1}{q_{-}(I_{2,k})}}$ and  $\mu
\big(I_{2,k}\big)^{\frac{1}{q_{+}'(I_{2,k})}} \approx
\|\chi_{I_{2,k}}\|_{L^{q'(\cdot)}(X)}\approx  \mu \big(
I_{2,k}\big)^{\frac{1}{q'_{-}(I_{k})}} $, where $k \leq m_0$.
Further, observe that these estimates and H\"{o}lder's inequality
yield the following chain of inequalities:
\begin{gather*}
J_{1} \leq c \sum\limits_{k\leq m_0}\;\; \int\limits_{\overline{B}(x_{0},A^{m_0+1})}\frac{\|f \chi_{I_{2,k}}\|_{L^{p(\cdot)}(X)}\cdot \|g \chi_{I_{2,k}}\|_{L^{q'(\cdot)}(X)}}{\|\chi_{I_{2,k}}\|_{L^{q(\cdot)}(X)}\cdot \|\chi_{I_{2,k}}\|_{L^{q'(\cdot)}(X)}}\chi_{E_{k}}(x)d\mu(x) \\
   =c \int\limits_{\overline{B}(x_{0},A^{m_0+1})}\sum\limits_{k\leq m_0}\frac{\|f \chi_{I_{2,k}}\|_{L^{p(\cdot)}(X)}\cdot \|g \chi_{I_{2,k}}\|_{L^{q'(\cdot)}(X)}}{\|\chi_{I_{2,k}}\|_{L^{q(\cdot)}(X)}\cdot \|\chi_{I_{2,k}}\|_{L^{q'(\cdot)}(X)}}\chi_{E_{k}}(x)d\mu(x) \\
\leq  c\Big\|\sum\limits_{k\leq  m_0}\frac{\|f \chi_{I_{2,k}}\|_{L^{p(\cdot)}(X)}}{\|\chi_{I_{2,k}}\|_{L^{q(\cdot)}(X)}}\chi_{E_{k}}(x) \Big\|
_{L^{q(\cdot)}(\overline{B}(x_{0},A^{m_0+1}))} \\
\times\Big\|\sum\limits_{k\leq
m_0}\frac{\|g
\chi_{I_{2,k}}\|_{L^{q'(\cdot)}(X)}}{\|\chi_{I_{2,k}}\|_{L^{q'(\cdot)}(X)}}
\chi_{E_{k}}(x)\Big\|_{L^{q'(\cdot)}(\overline{B}(x_{0},A^{m_0+1}))} =: c
S_{1}(f)\cdot S_{2}(g).
\end{gather*}

Now we claim that $  S_1(f) \leq c I(f) $, where
\begin{gather*}
I(f):=\Big\|\sum\limits_{k\leq m_0}\frac{\|f\chi_{I_{2,k}}\|_{L^{p(\cdot)}(X)}}{\|\chi_{I_{2k}}\|_{L^{p(\cdot)}(X)}}\chi_{E_{k}(\cdot)}\Big\|_{L^{p(\cdot)}(\overline{B}(x_0,A^{m_0+1}))}
\end{gather*}
and the positive constant $c$ does not depend on $f$. Indeed, suppose that $I(f)\leq 1$. Then  taking into account Lemma 1.11 we have that
\begin{gather*}
    \sum\limits_{k\leq m_0}\frac{1}{\mu (I_{2,k})}\int_{E_{k}}\|f\chi_{I_{2,k}} \|_{L^{p(\cdot)}(X)}^{p(x)}d\mu(x)\\
    \leq c\int\limits_{\overline{B}(x_0,A^{m_0+1})}\Big(\sum\limits_{k\leq m_0}\frac{ \|f\chi_{I_{2,k}}\|_{L^{p(\cdot)}(X)}}{\|\chi_{I_{2,k}}\|_{L^{p(\cdot)}(X)}}\chi_{E_{k}(x)}\Big)^{p(x)}d\mu(x) \leq c.
\end{gather*}
Consequently, since $p(x)\leq q(x)$, $E_k\subseteq I_{2,k}$ and $\|f\|_{L^{p(\cdot)}(X)}\leq 1$, we find that

$$   \sum\limits_{k\leq m_0}\!\frac{1}{\mu (I_{2,k})}\!\int\limits_{E_{k}}\!\!\|f\chi_{I_{2,k}}\|_{L^{p(\cdot)}(X)}^{q(x)}d\mu(x)\leq \sum\limits_{k\leq m_0}\frac{1}{\mu (I_{2,k})}\int\limits_{E_{k}}\!\!\|f\chi_{I_{2,k}}\|_{L^{p(\cdot)}(X)}^{p(x)}d\mu(x)\!\leq\! c. $$
This implies that $S_1(f)\leq c$. Thus  the desired inequality is proved. Further, let us introduce the following function:
$${\mathbb P}(y):=\sum\limits_{k\leq 2}p_+(\chi_{I_{2,k}})\chi_{E_{k}(y)}. $$

It is clear that $p(y)\leq{\mathbb P}(y)$  because $E_k \subset I_{2,k}.$ Hence
$$  I(f)\leq c \Big\|\sum\limits_{k\leq m_0}\frac{\|f\chi_{I_{2,k}}\|_{L^{p(\cdot)}(X)}}{\|\chi_{I_{2k}}\|_{L^{p(\cdot)}(X)}}\chi_{E_{k}(\cdot)}\Big\|_{L^{{\mathbb{P}}(\cdot)}(\overline{B}(x_0,A^{m_0+1}))}$$
for some positive constant $c$. Then by using the this inequality,
the definition of the function ${\mathbb {P}}$, the condition
$p\in LH(X)$ and the obvious estimate $
\|\chi_{I_{2,k}}\|_{L^{p(\cdot)}(X)}^{p_{+}(I_{2,k})}\geq c\mu
(I_{2,k})$, we find that
\begin{gather*} \int\limits_{\overline{B}(x_0,A^{m_0+1})}\!\!\!\!\!\!\bigg(\sum\limits_{k\leq m_0}\frac{\|f\chi_{I_{2,k}}\|_{L^{p(\cdot)}(X)}}{\|\chi_{I_{2,k}}\|_{L^{p(\cdot)}(X)}}\chi_{E_{k}(x)}\bigg)^{{\mathbb P}(x)}\!\!\!\!\!\!\!\!d\mu(x)\\
= \int\limits_{{\overline{B}}(x_0,A^{m_0+1})}\!\!\!\!\!\!\bigg(\sum \limits_{k\leq
m_0}\frac{\|f\chi_{I_{2,k}}\|_{L^{p(\cdot)}(X)}^{p_{+}(I_{2,k})}}{\|\chi_{I_{2,k}}\|_{L^{p(\cdot)}(X)}^{p_{+}(I_{2,k})}}\chi_{E_{k}(x)}\bigg)d\mu(x)\\
\leq c\int\limits_{\overline{B}(x_0,A^{m_0+1})}\bigg(\sum\limits_{k\leq m_0}\frac{\|f\chi_{I_{2,k}}\|_{L^{p(\cdot)}(X)}^{p_{+}(I_{2,k})}}{\mu (I_{2,k})}\chi_{E_{k}(x)}\bigg)d\mu(x)  \leq c\sum\limits_{k\leq
m_0}\|f\chi_{I_{2,k}}\|_{L^{p(\cdot)}(X)}^{p_{+}(I_{2,k})}\\
\leq c\sum\limits_{k\leq m_0}\int\limits_{I_{2,k}}|f(x)|^{p(x)}d\mu(x) \leq c \int\limits_{X}|f(x)|^{p(x)}d\mu(x)\leq c.
\end{gather*}

Consequently, $   I(f)\leq c \| f\|_{L^{p(\cdot)}(X)}$. Hence, $   S_1(f)\leq c \| f\|_{L^{p(\cdot)}(X)}$.   Analogously taking into account the fact that
$q'\in DL(X)$ and arguing as above we find that $ S_{2}(g) \leq c \| g\|_{L^{q'(\cdot)}(X)}$.
Thus summarizing these estimates  we conclude that
$$    \sum\limits_{i\leq m_0}\|f \chi_{I_{i}}\|_{L^{p(\cdot)}(X)}  \|g \chi_{I_{i}}\|_{L^{q'(\cdot)}(X)}\leq c \|f\|_{L^{p(\cdot)}(X)}\|g\|_{L^{q'(\cdot)}(X)}.  $$
\end{proof}

The next statement for metric measure spaces was proved in
\cite{HaHaPe} (see also \cite{KoMe3}, \cite{KoMe5} for  quasimetric
measure spaces).

\vskip+0.1cm
{\bf Theorem A.} {\em Let $(X, d,\mu)$ be an $SHT$ and let $\mu (X) < \infty$. Suppose that $1<p_-\leq
p_+<\infty$ and  $p\in {\mathcal{P}}(1)$. Then $M$ is bounded in
$L^{p(\cdot)}(X)$.}

For the following statement we refer to \cite{Kh}:
\vskip+0.1cm

{\bf Theorem B.}  {\em Let $(X, d, \mu)$ be an $SHT$ and let $L=\infty$. Suppose that $1<p_-\leq p_+<\infty$ and  $p\in {\mathcal{P}}(1)$.
Suppose also that $p=p_c= \hbox{const}$ outside some ball $B:= B(x_0,R)$. Then $M$ is bounded in $L^{p(\cdot)}(X)$.}

\section{Hardy--type transforms}

In this section we derive two-weight estimates for the operators:
$$   T_{v,w}f(x) = v(x)\int\limits _{B_{x_0 x}} f(y)w(y)d\mu(y)\;\; \hbox{and} \;\; T'_{v,w}f(x) = v(x)\int\limits _{X\backslash \overline{B}_{x_0 x}} f(y)w(y)d\mu(y). $$

%Recall that by $L$ we denote a diameter of $X$.

Let $a$ is a positive constant and let $p$ be a measurable
function defined on $X$. Let us introduce the notation:
$$
 p_{0}(x):=p_{-}(\overline{B}_{x_0x}); \;\;  \widetilde{p}_0 (x): = \left\{
  \begin{array}{ll}
    p_0 (x) & \hbox{if}\ \ d(x_0, x) \le a; \\
    p_c = \hbox{const} & \hbox{if}\ \   d(x_0, x) > a.
  \end{array}
\right. $$
$$
 p_{1}(x):=p_{-} \left(\overline{B}(x_0,a) \setminus B_{x_0x} \right); \,\, \widetilde{p}_1 (x) := \left\{
  \begin{array}{ll}
    p_1 (x) & \hbox{if}\ \ d(x_0, x) \le a; \\
    p_c = \hbox{const} & \hbox{if} \ \  d(x_0, x) > a.
  \end{array}
\right.
$$

\vskip+0.1cm {\em Remark} 2.1. If  we deal with a quasi-metric
measure space with $L<\infty$, then we will assume that $a=L$.
Obviously, $\widetilde{p}_0   \equiv p_0$ and  $\widetilde{p}_1
\equiv p_1$ in this case.

%Theorem 2.1
{\bf Theorem 2.1.} {\em Let $(X, d, \mu)$ be a quasi-metric
measure space  . Assume that $p$ and $q$ are measurable functions
on $X$ satisfying the condition $1<p_{-} \leq
\widetilde{p}_{0}(x)\leq q(x)\leq q_{+}<\infty. $  In the case
when $L=\infty$ suppose that $p\equiv p_c\equiv$ const, $q\equiv
q_c\equiv$ const,  outside some ball $\overline{B}(x_0,a)$. If the
condition
$$    A_{1}:= \sup\limits_{0\leq t\leq L} \int\limits_{t< d(x_{0},x)\leq L}\big(v(x)\big)^{q(x)}\bigg(\int\limits_{d(x_{0},x)\leq t}
    w^{(\widetilde{p}_{0})'(x)}(y)d\mu(y)\bigg)^{\frac{q(x)}{(\widetilde{p}_0)'(x)}}d\mu(x)<\infty,$$
 hold, then $T_{v,w}$ is bounded from $L^{p(\cdot)}(X)$ to  $ L^{q(\cdot)}(X)$.}
 \vskip+0.1cm

{\em Proof.} Here we use the arguments of the proofs of Theorem
1.1.4 in \cite{EdKoMe} (see p. 7) and of Theorem 2.1 in
\cite{EdKoMe1}. First we notice that $p_{-} \leq p_{0}(x) \leq
p(x)$  for all $ x \in X $. Let $f \geq 0$ and let $S_{p}(f)\leq
1$.  First assume that $L<\infty$. We denote
$$     I(s):= \int\limits_{d(x_{0},y)<s} f(y) w(y) d\mu(y) \ \ \hbox{for}\ s \in [0,L]. $$
Suppose that $I(L)<\infty$.  Then $I(L) \in (2^{m},2^{m+1}]$ for
some $ m \in \mathbb{Z}. $ Let us denote $   s_{j}:= \sup
\{s:I(s)\leq 2^{j}\}, \ j \leq m$, and  $s_{m+1}:=L. $ Then
$\big\{ s_{j}\big\}_{j=-\infty}^{m+1}$ is a non-decreasing
sequence. It is easy to check that $I(s_{j})\leq 2^{j},\ I(s)>
2^{j}$ for $s>s_{j}$, and $    2^{j}\leq \int\limits_{s_{j}\leq
d(x_{0},y)\leq s_{j+1}}f(y) w(y) d\mu (y) $. If $\beta:=
\lim\limits_{j\rightarrow -\infty}s_{j},$ then $ d(x_{0},x)<L$ if
and only if  $d(x_{0},x)\in [0,\beta]\cup \bigcup
\limits_{j=-\infty}^{m}(s_{j},s_{j+1}].  $ If $I(L)=\infty$ then
we take $m=\infty$. Since $  0 \leq I(\beta) \leq I(s_{j}) \leq
2^{j} $ for every  $j$, we have that $  I(\beta) = 0. $ It is
obvious that $ X =\bigcup \limits_{j\leq m}\{x: s_{j}< d(x_{0},x)
\leq s_{j+1}\}$. Further, we have that
\begin{eqnarray*}
  S_q(T_{v,w}f) = \int\limits_X (T_{v,w}f(x))^{q(x)} d\mu(x) = \int\limits_X \Bigg( v(x) \!\!\!\!\!\!\! \int\limits_{B(x_0,\ d(x_0,x))}
f(y) w(y)d\mu(y) \Bigg)^{q(x)} \!\!\!\! d\mu(x)\\
 = \int\limits_X (v(x))^{q(x)} \Bigg( \int\limits_{B(x_0,\ d(x_0,x))}
f(y) w(y)d\mu(y) \Bigg)^{q(x)} d\mu(x) \\
 \leq  \sum\limits_{j=-\infty}^m \int\limits_{s_j < d(x_0,x)\leq s_{j+1}}\!\!\!\!\!\!\!\!\!\!\Big(v(x)\Big)^{q(x)}
\!\Bigg(\;\;\;\; \int\limits_{d(x_0,y)< s_{j+1}}
\!\!\!\!\!\!\!\!\! f(y) w(y)d\mu(y) \Bigg)^{q(x)} d\mu(x).
\end{eqnarray*}

Notice that $ I(s_{j+1}) \leq 2^{j+1} \leq 4
\int\limits_{s_{j-1}\leq d(x_{0},y)\leq s_{j}}w(y) f(y) d\mu(y) $.
Consequently, by this estimate  and H\"older's inequality with
respect to the exponent $p_0(x)$ we find that
\begin{eqnarray*}
     S_{q}\big(T_{v,w}f\big) \leq c \sum \limits_{j=-\infty}^{m}\int\limits_{s_{j}<d(x_{0},x)\leq s_{j+1}}\!\!\!\!\!\!\!\!\!\Big(v(x)\Big)^{q(x)}\Bigg(
    \int\limits_{s_{j-1}\leq d(x_{0},y)\leq s_{j}}\!\!\!\!\!\!\!\!\!f(y)w(y)d \mu (y)\Bigg)^{q(x)}d\mu(x) \\
      \leq c\sum\limits_{j=-\infty}^{m}\int\limits_{s_{j}<d(x_{0},x)\leq s_{j+1}}\!\!\!\!\!\!\!\!\!\big(v(x)\big)^{q(x)}J_{k}(x)d\mu(x),
 \end{eqnarray*}
where
$$ J_{k}(x):=\bigg(\!\!\!\!\!\!\!\!\!\int\limits_{s_{j-1}\leq d(x_{0},y)\leq s_{j}}\;\;\!\!\!\!\!\!\!\!\!\!\!f(y)^{p_{0}(x)}d\mu(y)\bigg)^{\frac{q(x)}{p_{0}(x)}}       \bigg(\!\!\!\!\!\!\!\!\!\int\limits_{s_{j-1}\leq d(x_{0},y)\leq s_{j}}\!\!\!\!\!\!\!\!\!w(y)^{(p_{0})'(x)}d \mu (y)\bigg)^{\frac{q(x)}{(p_{0})'(x)}}. $$

Observe now that  $ q(x) \geq p_{0}(x)$. Hence, this fact and the
condition $S_{p}(f)\leq 1 $ imply that
\begin{eqnarray*}
J_{k}(x) \leq c \bigg(\int\limits_{\{y:s_{j-1}\leq d(x_{0},y)\leq s_{j}\}\cap \{y:f(y)\leq 1\}}\!\!\!\!\!\!\!\!\!\!\!\!\!\!\!\!\!\!\!f(y)^{p_{0}(x)}d\mu(y) +  \!\!\!\!\!\!\!\!\int\limits_{\{y:s_{j-1}\leq d(x_{0},y)\leq s_{j}\}\cap \{y:f(y)> 1\}}\!\!\!\!\!\!\!\!\!\!\!\!\!\!\!\!\!\!\!\!\!\!\!\!\!\! f(y)^{p(y)}d\mu(y)\bigg)^{\frac{q(x)}{p_{0}(x)}} \\
\times \bigg(\int\limits_{s_{j-1}\leq d(x_{0},y)\leq s_{j}}\!\!\!\!\!\!\!\!\!w(y)^{(p_{0})'(x)}d \mu (y)\bigg)^{\frac{q(x)}{(p_{0})'(x)}} \\
 \leq  c \bigg( \mu \big( \{y:s_{j-1}\leq d(x_{0},y)\leq s_{j}\}\big) +\int\limits_{\{y:s_{j-1}\leq d(x_{0},y)\leq s_{j}\}\cap \{y:f(y)> 1\}}\!\!\!\!\!\!\!\!\!\!\!\!\!\!\!f(y)^{p(y)}d\mu(y)\bigg)\\
\times \bigg(\int\limits_{s_{j-1}\leq d(x_{0},y)\leq
s_{j}}w(y)^{(p_{0})'(x)}d \mu
(y)\bigg)^{\frac{q(x)}{(p_{0})'(x)}}.
\end{eqnarray*}

It follows now  that

\begin{eqnarray*}
 S_q(T_{v,w}f)  \leq c \bigg(\sum\limits_{j=-\infty}^{m}\mu \big(\{y:s_{j-1}\leq d(x_{0},y)\leq s_{j}\}\big)\int\limits_{s_{j}<d(x_{0},x)\leq s_{j+1}}v(x)^{q(x)}\\
  \times \bigg( \int\limits_{s_{j-1}\leq d(x_{0},y)\leq s_{j}}w(y)^{(p'_{0})(x)}d\mu(y)\bigg)^{\frac{q(x)}{(p_{0})'(x)}} d\mu(x)\\
+ \sum\limits_{j=-\infty}^{m}\bigg(\int\limits_{y:\{s_{j-1}\leq d(x_{0},y)\leq s_{j}\}\cap\{y: f(y)>1\}}f(y)^{p(y)}d\mu(y)\bigg) \\
\times \!\!\!\!\! \int\limits_{s_{j}<d(x_{0},x)\leq
s_{j+1}}\!\!\!\!\!\!\!\!\! v(x)^{q(x)} \bigg(\!\!\!\!
\int\limits_{s_{j-1}\leq d(x_{0},y)\leq
s_{j}}w(y)^{(p_{0})'(x)}d\mu(y)\bigg)^{\frac{q(x)}{(p_{0})'(x)}}d\mu(x)
\bigg) :=c\big(N_{1}+N_{2}\big).
\end{eqnarray*}
 It is obvious that
 $$ N_{1} \leq A_{1}\sum\limits_{j=-\infty}^{m+1}\!\!\mu \big( \{y:s_{j-1}\leq d(x_{0},y)\leq s_{j}\}\big) \leq  C A_{1}  $$
and
$$ N_{2} \leq A_{1}\sum\limits_{j=-\infty}^{m+1}\int\limits_{\{y:s_{j-1}\leq d(x_{0},y)\leq s_{j}\}}\!\!\!\!\!\!\!\!\!\!\!\!\!\!\!f(y)^{p(y)}d\mu(y)= C \int\limits_{X}\big(f(y)\big)^{p(y)}d\mu(y)=A_{1}S_{p}(f) \leq A_{1}. $$

Finally $  S_q(T_{v,w}f) \leq c\big(c A_{1}+A_{1}\big)<\infty $.
Thus $T_{v,w}$ is bounded if $A_{1}< \infty$.

Let us now suppose that $L=\infty$.  We have
\begin{eqnarray*}
   T_{v,w}f(x) = \chi_{B(x_0,a)}(x) v(x)\int\limits_{B_{x_0x}} f(y)w(y)d\mu(y) \\
   +\chi_{X \backslash B(x_0,a)}(x)v(x)  \int\limits_{B_{x_0x}}f(y)w(y)d\mu(y)
=: T_{v,w}^{(1)}f(x)+T_{v,w}^{(2)}f(x)
\end{eqnarray*}

By using already proved result for $L<\infty$ and the fact  that
${\hbox{diam}}\; \big(B(x_0,a)\big)<\infty$  we find that $
\|T_{v,w}^{(1)}f\|_{L^{q(\cdot)}\big(B(x_0,a)\big)}\leq
c\|f\|_{L^{p(\cdot)}\big(B(x_0,a)\big)} \leq c$ because
$$  A_{1}^{(a)}:= \sup\limits_{0\leq t \leq a}\int\limits_{t< d(x_{0},x)\leq a}\!\!\!\!\!\!\big(v(x)\big)^{q(x)}\bigg(\int\limits_{d(x_{0},x)\leq t}
    w^{(p_{0})'(x)}(y)d\mu(y)\bigg)^{\frac{q(x)}{(p_{0})'(x)}}\!\!\!\!\!d\mu(x)  \leq A_1 <\infty. $$
Further,  observe that
\begin{eqnarray*}
T_{v,w}^{(2)}f(x)\!\!  =\!\!  \chi_{X\backslash B(x_{0},a)}(x)v(x)
\!\!\!\int\limits_{B_{x_0x}}\!\!\!f(y)w(y)d\mu(y) =
\chi_{X\backslash B(x_{0},a)} (x)v(x)
\!\!\!\!\!\!\!\! \int\limits_{d(x_{0},y)\leq a}  \!\!\!\!\! f(y)w(y)d\mu(y)\\
+\chi_{X\backslash B(x_{0},a)}(x)v(x) \!\!\!\!\!\!
\int\limits_{a\leq d(x_{0},y)\leq d(x_{0},x)}     \!\!\!\!\!\!
f(y)w(y)d\mu(y) =:  T^{(2,1)}_{v,w} f(x)+T^{(2,2)}_{v,w} f(x).
\end{eqnarray*}

It is easy to see that (see also Theorem 1.1.3 or 1.1.4 of
\cite{EdKoMe}) the condition
$$\overline A_{1}^{(a)}:=\sup\limits_{t\geq a} \bigg(\int\limits_{d(x_0,x)\geq t} \!\!\!\! \big(v(x)\big)^{q_{c}}d\mu(x)\bigg)^{\frac{1}{q_{c}}}\bigg(\!\!\!\int\limits_{a\leq d(x_0,y)\leq t}\!\!\!\!\!\!\!\! w(y)^{(p_{c})'}d\mu(y)\bigg)^{\frac{1}{(p_{c})'}} <\infty $$
guarantees the boundedness  of the operator
$$T_{v,w} f(x)= v(x)\!\!\!\!\!\!\! \int\limits_{a\leq d(x_0,y)< d(x_0,x)}\!\!\!\! f(y)w(y)d\mu(y)$$
from $L^{p_{c}}\big(X\backslash B(x_0,a)\big)$  to
$L^{q_{c}}\big(X\backslash B(x_0,a)\big).$ Thus  $T^{(2,2)}_{v,w}$
is  bounded. It remains to prove that $T^{(2,1)}_{v,w}$  is
bounded. We have
$$
 \| T^{(2,1)}_{v,w}f \|_{L^{p(\cdot)}(X)} = \Bigg(\int\limits_{\big(B(x_0,a)\big)^{c}} v(x)^{q_{c}} d\mu(x)\Bigg)^{\frac{1}{q_{c}}} \Bigg(\int\limits_{\overline B(x_0,a)}f(y)w(y)d\mu(y)\Bigg)$$
 $$
  \leq  \Bigg(\int\limits_{\big(B(x_0,a)\big)^{c}} v(x)^{q_{c}} d\mu(x)\Bigg)^{\frac{1}{q_{c}}} \|f\|_{L^{p(\cdot)}\big(\overline B(x_0,a)\big)}\|w\|_{L^{p'(\cdot)}\big(\overline B(x_0,a)\big)}.
 $$
Observe now that the condition $A_1<\infty$ guarantees that the
integral
$$ \int\limits_{\big(B(x_0,a)\big)^{c}} v(x)^{q_{c}} d\mu(x)$$
is   finite. Moreover,   $N:= \|w\|_{L^{p'(\cdot)}\big( \overline
B(x_0,a)\big)}<\infty$. Indeed, we have that
$$N \leq   \left\{   \begin{array}{ll}
    \bigg(\int\limits_{\overline B(x_0,a)}w(y)^{p'(y)}d\mu(y)\bigg)^{\frac{1}{\big(p_- (\overline B(x_0,a))\big)'}} & \hbox{if} \ \|w\|_{L^{p'(\cdot)}(\overline{B}(x_0,a))}\leq 1, \\
     \bigg(\int\limits_{\overline B(x_0,a)}w(y)^{p'(y)}d\mu(y)\bigg)^{\frac{1}{\big(p_+ (\overline B(x_0,a))\big)'}}   & \hbox{if}\ \|w\|_{L^{p'(\cdot)}(\overline{B}(x_0,a)}> 1.
  \end{array}
\right. $$ Further,
$$ \int\limits_{\overline B(x_0,a)}\!\!\! w(y)^{p'(y)}d\mu(y) =   \!\!\!\!\!\!\int\limits_{\overline B(x_0,a)\cap{\{w\leq 1\}}}\!\!\!\!\!\!w(y)^{p'(y)}d\mu(y)+   \!\!\!\!\!\!\!\!\!\int\limits_{\overline B(x_0,a)\cap{\{w> 1\}}}\!\!\!\!\!\!w(y)^{p'(y)}d\mu(y)   := I_{1}+I_{2}. $$
For $I_{1}$, we have that $ I_{1} \leq \mu\big( \overline
B(x_0,a))<\infty$. Since $L=\infty$ and condition (1) holds, there
exists a point $y_0\in X$ such that $a< d(x_0,y_0)<2a$.
Consequently, $\overline{B}(x_0,a) \subset \overline{B}(x_0,
d(x_0,y_0))$ and $    p(y) \geq p_{-} \big({ \overline
B(x_0,d(x_0,y_0))}\big) =p_{0}(y_0) $, where $y\in
\overline{B}(x_0, a)$. Consequently, the condition $A_1<\infty$
yields $I_2  \leq \int\limits_{\overline B(x_0,a)
}w(y)^{(p_0)'(y_0)}dy  <\infty $. Finally we have that $ \|
T^{(2,1)}_{v,w}f \|_{L^{p(\cdot)}(X)} \leq C$. Hence, $T_{v,w}$ is
bounded from $L^{p(\cdot)}(X)$ to $L^{q(\cdot)}(X)$. $\Box$

The proof of the following statement is similar to that of Theorem
2.1;  therefore we omit it (see also the proofs of Theorem 1.1.3
in \cite{EdKoMe} and Theorems 2.6 and 2.7 in \cite{EdKoMe1} for
similar arguments). \vskip+0.1cm

%Theorem 2.3
{\bf Theorem 2.2.} {\em  Let  $(X, d, \mu)$ be a quasi-metric
measure space  . Assume that $p$ and $q$ are measurable functions
on $X$ satisfying the condition $1<p_{-}\leq
\widetilde{p}_1(x)\leq q(x)\leq q_{+}<\infty. $ If $L= \infty$,
then we assume that $p\equiv p_c\equiv$ const, $q\equiv q_c\equiv$
const outside some ball $B(x_0,a)$.
 If
$$    B_{1}= \sup\limits_{0\leq t \leq L} \int\limits_{d(x_{0},x)\leq t}\big(v(x)\big)^{q(x)}\bigg(\int\limits_{t\leq d(x_{0},x)\leq L}
    w^{(\widetilde{p}_{1})'(x)}(y)d\mu(y)\bigg)^{\frac{q(x)}{(\widetilde{p}_1)'(x)}}d\mu(x)<\infty,$$
then $T'_{v,w}$ is bounded from $L^{p(\cdot)}(X)$ to$
L^{q(\cdot)}(X).$}

\vskip+0.1cm {\em Remark} 2.2. If $p\equiv$ const, then the
condition $A_1<\infty$ in Theorem 2.1 (resp. $B_1<\infty$ in
Theorem 2.2) is also necessary for the boundedness of $T_{v,w}$
(resp. $T'_{v,w}$) from $L^{p(\cdot)}(X)$ to $L^{q(\cdot)}(X)$.
See \cite{EdKoMe},  pp.4-5, for the details.

\section{Potentials}

In this section we discuss two--weight estimates for the potential
operators $T_{\alpha(\cdot)}$ and $I_{\alpha(\cdot)}$ on
quasi-metric measure spaces, where $0<\alpha_- \leq \alpha_+ <1$.
If $\alpha\equiv \;{\hbox{const}}$, then  we denote
$T_{\alpha(\cdot)}$ and $I_{\alpha(\cdot)}$ by $T_{\alpha}$ and
$I_{\alpha}$ respectively.

The boundedness of Riesz potential operators in
$L^{p(\cdot)}(\Omega)$ spaces, where $\Omega$ is a domain in
${\Bbb{R}}^n$ was established in \cite{Di2}, \cite{Sa2},
\cite{CrFiMaPe}, \cite{CaCrFi}.

For the following statement we refer to \cite{KoSa5}: \vskip+0.1cm

{\bf Theorem C.} {\em Let $(X, d,\mu)$ be an $\hbox{SHT}$. Suppose
that $1<p_-\leq p_+<\infty$ and $p\in {\mathcal{P}}{(1)}$. Assume
that if $L=\infty$, then $p\equiv  \; const$ outside some ball.
Let $\alpha$ be  a constant satisfying the condition $0< \alpha <
1/p_+$.  We set $q(x)= \frac{p(x)}{1-\alpha p(x)}$. Then
$T_{\alpha}$ is bounded in $L^{p(\cdot)}(X)$.} \vskip+0.1cm

{\bf Theorem D \cite{KoMe5}.} {\em Let  $(X, d, \mu)$ be a
non--homogeneous space with $L<\infty$ and let $N$ be a constant
defined by $N=a_1(1+2 a_0)$, where the constants $a_0$ and $a_1$
are taken from the definition of the quasi--metric $d$. Suppose
that  $1<p_-<p_+<\infty$, $p, \alpha \in {\mathcal{P}}(N)$ and
that $\mu$ is upper Ahlfors $1$-regular. We define $q(x)=
\frac{p(x)}{1-\alpha(x)p(x)}$, where $0< \alpha_-\leq \alpha_+<
1/p_-$. Then $I_{\alpha(\cdot)}$ is bounded from $L^{p(\cdot)}(X)$
to $L^{q(\cdot)}(X)$.} \vskip+0.1cm
%For Theorem D, in the case when $\mu$ satisfies the doubling condition, we refer also to \cite{MuSa}.

For the statements and their proofs of this section we keep the
notation of the previous sections and, in addition,  introduce
the  new notation:
\begin{eqnarray*}
v^{(1)}_{\alpha}(x):=v(x)(\mu B_{x_{0} x})^{\alpha-1},\;\;
w^{(1)}_{\alpha}(x):=w^{-1}(x);\
v^{(2)}_{\alpha}(x):=v(x);\\
 w^{(2)}_{\alpha}(x):=w^{-1}(x)(\mu B_{x_{0} x})^{\alpha-1};\\
F_x:=  \begin{array}{ll} \{ y\in X: \frac{d(x_0,y) L}{A^2a_1} \leq d(x_0, y) \leq A^2 L a_1 d(x_0,x)\}, \;\; \hbox{if} \;\; L<\infty \\
\{ y\in X: \frac{d(x_0,y)}{A^2a_1} \leq d(x_0, y) \leq A^2  a_1
d(x_0,x)\}, \;\; \hbox{if}\;\; L= \infty,
 \end{array},
\end{eqnarray*}
where $A$ and $a_1$ are constants defined in Definition 1.10  and
the triangle inequality for $d$ respectively.
 We begin  this section with the following general--type statement:

\vskip+0.1cm

 %Theorem 3.1
{\bf Theorem 3.1.} {\em Let $(X, d, \mu)$ be an $\hbox{SHT}$
without atoms. Suppose that $1<p_-\leq p_+<\infty$ and $\alpha$ is
a constant satisfying the condition $0<\alpha<1/p_+$. Let $p \in
{\mathcal{P}}(1)$. We set $q(x)=\frac{p(x)}{1-\alpha p(x)}$.
Further, if $L=\infty$, then we assume that   $p\equiv
p_{c}\equiv$ const outside some ball  $B(x_{0},a)$. Then the
inequality

$$ \| v (T_{\alpha}f) \|_{L^{q(\cdot)}(X)} \leq c \| wf\|_{L^{p(\cdot)}(X)} \eqno{(5)}$$
holds if the following three conditions are satisfied:

$(a)\;\;\;  T_{v^{(1)}_{\alpha},w^{(1)}_{\alpha}} $  is bounded
from $L^{p(\cdot)}(X)$ to $L^{q(\cdot)}(X)$ ;

$(b)\;\;\;  T_{v^{(2)}_{\alpha},w^{(2)}_{\alpha}}$  is bounded
from $L^{p(\cdot)}(X)$ to $L^{q(\cdot)}(X)$;

$(c)\;$ there is a positive constant $b$ such that one of the
following inequality holds:  $1) \; v_+(F_x)  \leq b w(x)$  for
$\mu-$ a.e. $x\in X\;\;\; $;  $2)\;
 v(x)\leq b  w_-(F_x) $ for $\mu-$ a.e. $x\in X. $}
\vskip+0.1cm

{\em Proof.}  For simplicity suppose that $L< \infty$. The proof
for the case $L=\infty$ is similar to that of the previous case.
Recall  that the sets $I_{i,k}$, $i=1,2,3$ and $E_k$ are defined
in Section 1. Let $f\geq 0$ and let $\|g\|_{L^{q'(\cdot)}(X)} \leq
1$. We have

\begin{eqnarray*}
\int\limits_{X} (T_{\alpha}f)(x) g(x) v(x) d\mu(x) = \sum_{k= -\infty}^0 \int\limits_{E_k} (T_{\alpha}f)(x) g(x) v(x) d\mu(x) \\
\leq \sum_{k=-\infty}^{0} \int\limits_{E_k} (T_{\alpha}f_{1,k})(x)
g(x) v(x) d\mu(x)+ \sum_{k=-\infty}^0 \int\limits_{E_k}
(T_{\alpha}f_{2,k})(x)
g(x) v(x) d\mu(x) \\
 + \sum_{k=-\infty}^0 \int\limits_{E_k} (T_{\alpha}f_{3,k}) (x) g(x) v(x) d\mu(x):= S_1+ S_2 +S_3,
\end{eqnarray*}
where $ f_{1,k}= f\cdot \chi_{I_{1,k}}$,  $f_{2,k}= f\cdot
\chi_{I_{2,k}}$, $f_{3,k}= f\cdot \chi_{I_{3,k}}. $

Observe that if $x\in E_k$ and $y\in I_{1,k}$, then $d (x_0, y)
\leq d(x_0,x)/Aa_1$. Consequently, the triangle inequality for $d$
yields $d(x_0,x)\leq A' a_1a_0 d(x,y)$, where $A'=A/(A-1)$. Hence,
by using Remark 1.1 we find that $ \mu(B_{x_0x})\leq c \mu(
B_{xy})$. Applying now condition (a)  we have that

$$ S_1 \leq c \bigg\| \big(\mu B_{x_0 x}\big)^{\alpha-1} v(x) \int\limits_{B_{x_0 x}} f(y) d\mu(y) \bigg\|_{L^{q(x)}(X)}
\| g \|_{L^{q'(\cdot)}(X)} \leq c \| f \|_{L^{p(\cdot)}(X)}. $$

Further,  observe that if $x\in E_k$ and $y\in I_{3,k}$, then $
\mu \big(B_{x_0y}\big) \leq c \mu\big( B_{xy}\big)$. By condition
(b) we find that $ S_3 \leq c \| f \|_{L^{p(\cdot)}(X)} $.

Now we estimate $S_2$.  Suppose that $v_+(F_x) \leq b w(x)$.
Theorem C and Lemma 1.12 yield
\begin{eqnarray*}
 S_2 \leq \sum_{k} \| \big(T_{\alpha} f_{2,k}\big)(\cdot) \chi_{E_k}(\cdot) v(\cdot) \|_{L^{q(\cdot)}(X)} \|g \chi_{E_k}(\cdot)\|_{L^{q'(\cdot)}(X)} \\ \leq \sum_{k} \Big( v_+( E_k) \Big)\| (T_{\alpha} f_{2,k})(\cdot)  \|_{L^{q(\cdot)}(X)} \| g(\cdot) \chi_{E_k}(\cdot) \|_{L^{q'(\cdot)}(X)} \\
  \leq  c \sum_{k} \Big( v_+(E_k)  \Big) \| f_{2,k} \|_{L^{p(\cdot)}(X)} \| g(\cdot) \chi_{E_k}(\cdot) \|_{L^{q'(\cdot)}(X)} \\
   \leq  c \sum_{k}  \| f_{2,k}(\cdot) w(\cdot) \chi_{I_{2,k}}(\cdot) \|_{L^{p(\cdot)}(X)} \| g(\cdot)  \chi_{E_k}(\cdot) \|_{L^{q'(\cdot)}(X)} \\
    \leq  c \| f(\cdot) w(\cdot)  \|_{L^{p(\cdot)}(X)} \| g(\cdot)  \|_{L^{q'(\cdot)}(X)} \leq  c \| f(\cdot) w(\cdot)  \|_{L^{p(\cdot)}(X)}.
\end{eqnarray*}

The estimate of $S_2$ for the case when $v(x) \leq b w_-(F_{x})$
is similar  to that of the previous one. Details are omitted.
$\Box$

\vskip+0.1cm

Theorems 3.1, 2.1 and 2.2 imply the following statement:
\vskip+0.1cm

%Theorem 3.2
{\bf Theorem 3.2.} {\em Let $(X, d, \mu)$ be an $\hbox{SHT}$.
Suppose that $1<p_-\leq p_+<\infty$ and $\alpha$ is a constant
satisfying the condition $0<\alpha<1/p_+$. Let $p \in
{\mathcal{P}}(1)$. We set $q(x)=\frac{p(x)}{1-\alpha p(x)}$. If
$L=\infty$, then we suppose that   $p\equiv p_{c}\equiv$ const
outside some ball  $B(x_{0},a)$.  Then inequality $(5)$ holds if
the following three conditions are satisfied:

$$(i)\;\;\; P_1\!:=\!\!  \sup\limits_{0<t\leq  L}\!\!\!\!\! \int\limits_{t< d(x_{0},x)\leq L}\!\!\!\!\!
\bigg(\frac{v(x)}{\big(\mu (B_{x_0
x})\big)^{1-\alpha}}\bigg)^{q(x)}
\!\bigg(\!\!\!\!\!\int\limits_{d(x_{0},y)\leq t}
\!\!\!\!\!\!\!w^{-({\widetilde{p}}_{0})'(x)}(y)d\mu(y)\bigg)^{\frac{q(x)}{({\widetilde{p}}_{0})'(x)}}\!\!\!
d\mu(x)\! <\!\infty;$$

$$ (ii)\;\;\; P_2\!:=\!\!\!\!     \sup\limits_{0<t\leq  L}\!\!\!\!\!\int\limits_{d(x_{0},x)\leq t}\!\!\!\!\!\!\!\big(v(x)\big)^{q(x)} \bigg(\!\!\!\!\!\!\!\int\limits_{t< d(x_{0},y)\leq L}
\!\!\!\!\!\!\!\!\!\!\! \Big( w(y)\big(\mu B_{x_0y}\big)^{1-\alpha}
\Big)^{-({\widetilde{p}}_{1})'(x)}\!\! d\mu(y)
\bigg)^{\frac{q(x)}{({\widetilde{p}}_1)'(x)}}\!\!\!d\mu(x)\!
<\!\infty,
$$

$(iii)\;\;\;\; $ condition $(c)$ of Theorem $3.1$ holds.}
\vskip+0.1cm

{\em Remark} 3.1. If $p= p_c\equiv $ const on $X$, then the
conditions $P_i<\infty$, $i=1,2$, are necessary for (5). Necessity
of the condition $P_1<\infty$ follows by taking the test function
$f= w^{-(p_c)'} \chi_{B(x_0,t)}$ in  (5) and observing that $\mu
B_{xy}\leq c\mu B_{x_0x}$ for those $x$ and $y$ which satisfy the
conditions $d(x_0,x)\geq t$ and $d(x_0,y)\leq t$ (see also
\cite{EdKoMe}, Theorem 6.6.1, p. 418 for the similar arguments),
while necessity of the condition $P_2<\infty$ can be derived by
choosing the test function $f(x)= w^{-(p_c)'}(x) \chi_{X\setminus
B(x_0,t)}(x)\big( \mu B_{x_0x}\big)^{(\alpha-1)((p_c)'-1)}$ and
taking into account the estimate $\mu B_{xy}\leq \mu B_{x_0y}$ for
$d(x_0,x)\leq t$ and $d(x_0,y)\geq t$. \vskip+0.1cm

The next statement follows in the same manner as the previous one.
In this case Theorem D is used instead of Theorem C. The proof is
omitted.

\vskip+0.1cm

%Theorem 3.3.
{\bf Theorem 3.3.} {\em  Let  $(X, d, \mu)$ be a non--homogeneous
space with $L<\infty$. Let $N$ be a constant defined by $N=a_1(1+2
a_0)$. Suppose that  $1<p_-\leq p_+<\infty$, $p, \alpha \in
{\mathcal{P}}(N)$ and that $\mu$ is upper Ahlfors $1$-regular. We
define $q(x)= \frac{p(x)}{1-\alpha(x)p(x)}$, where $0<
\alpha_-\leq \alpha_+< 1/p_+$. Then  the inequality

$$ \| v(\cdot) (I_{\alpha(\cdot)}f)(\cdot) \|_{L^{q(\cdot)}(X)} \leq c \| w(\cdot) f(\cdot)\|_{L^{p(\cdot)}(X)} \eqno{(6)}$$ holds if

 $$(i)\;\;\;
 \sup\limits_{0\leq t\leq L}\!\!\!\!\!\! \int\limits_{t< d(x_{0},x)\leq L} \!\!\!\!\!\!\!\bigg(\frac{v(x)}{\big(d(x_0,x)\big)^{1-\alpha(x)}}\bigg)^{q(x)} \bigg(\! \int\limits_{\overline{B}(x_0,t)} \!\!\!\!w^{-(p_{0})'(x)}(y)d\mu(y)\bigg)^{\frac{q(x)}{(p_{0})'(x)}}\!\!\!\! d\mu(x)\!<\!\infty;$$

$$ (ii)\;\;\; \sup\limits_{0\leq t \leq L}\!\!\!\int\limits_{\overline{B}(x_0,t)}\!\!\!\!\!\!\big(v(x)\big)^{q(x)}\bigg(\!\!\!\int\limits_{t< d(x_{0},y)\leq L}
\!\!\!\!\!\!\! \big(w(y) d(x_0,y)
^{1-\alpha(y)}\big)^{-(p_{1})'(x)}d\mu(y)\bigg)^{\frac{q(x)}{(p_1)'(x)}}\!\!\!\!
d\mu(x)\!<\!\infty,
$$

$(iii)\;\;\; $ condition $(c)$ of Theorem $3.1$ is satisfied.}
\vskip+0.1cm

{\em Remark} 3.2. It is easy to check that if $p$ and $\alpha$ are
constants, then conditions (i) and (ii) in Theorem 3.3 are also
necessary for (6). This follows easily by choosing appropriate
test functions in $(6)$ (see also Remark 3.1) \vskip+0.1cm

%Theorem 3.4
{\bf Theorem 3.4.} {\em  Let $(X, d, \mu)$ be an $\hbox{SHT}$
without atoms. Let $1<p_-\leq p_+<\infty$ and let $\alpha$ be a
constant with  the condition $0<\alpha <1/p_+$. We set
$q(x)=\frac{p(x)}{1-\alpha p(x)}$.  Assume  that $p$ has a minimum
at $x_0$ and  that $p \in \hbox{LH}(X)$. Suppose also that if $L=
\infty$, then $p$ is constant outside some ball $B(x_0,a)$. Let
$v$ and $w$ be positive increasing functions on $(0,2L)$. Then the
inequality
$$ \| v(d(x_0,\cdot)) (T_{\alpha} f) (\cdot) \|_{L^{q(\cdot)}(X)} \leq c \| w(d(x_0,\cdot))f(\cdot)\|_{L^{p(\cdot)}(X)} \eqno{(7)}$$
holds if
$$  I_1\!:=  \!\!\!\sup_{0 < t \leq L} \!\!\!I_1(t) \!\!:= \!\!\!\sup\limits_{0< t \leq L}\!\!\!\!\int\limits_{t< d(x_{0},x)\leq L}\!\!\!\!\bigg( \frac{v(d(x_0,x))}{\big(\mu (B_{x_0 x} )\big)^{1-\alpha}}\bigg)^{q(x)}$$
 $$ \times \bigg(\!\!\!\!\!\!\!\!\int\limits_{d(x_{0},y)\leq t}
\!\!\!\!\!\!\!\!\!
w^{-({\widetilde{p}}_{0})'(x)}(d(x_0,y))d\mu(y)\!\!
\bigg)^{\frac{q(x)}{({\widetilde{p}}_{0})'(x)}}\!\!d\mu(x) <
\infty $$ for $  L=\infty;$

$$
 J_1\!\!:=\!\!\! \sup\limits_{0<t\leq L}\!\!\!\!\!\! \int\limits_{t< d(x_{0},x)\leq L}\!\!\!\!\!\!\!\bigg(\frac{v(d(x_0,x))}{\big(\mu (B_{x_0 x} )\big)^{1-\alpha}}\bigg)^{q(x)}\!\!\!
 \bigg(\!\!\!\!\!\!\!\int\limits_{d(x_{0},y)\leq t} \!\!\!\!\!\!\!\!\!
w^{-p'(x_0)}(d(x_0,y))d\mu(y)\bigg)^{\frac{q(x)}{p'(x_0)}}\!\!\!d\mu(x)\!<\!\infty
$$
for $ L<\infty.$} \vskip+0.1cm

{\em Proof.} Let $L=\infty$. Observe that by Lemma 1.9 the
condition $p\in \hbox{LH}(X)$ implies  $p\in {\mathcal{P}}(1)$. We
will show that the condition $I_1<\infty$ implies the inequality $
\frac{v(A^2a_1t)}{w(t)} \leq C $ for all $t>0$, where $A$ and
$a_1$ are constants defined in Definition 1.10 and the triangle
inequality for $d$ respectively. Indeed, let us assume that $t\leq
b_1$, where $b_1$ is a small positive constant. Then, taking into
account the monotonicity of $v$ and $w$, and the facts that
${\widetilde{p}}_0(x) = p_0(x)$ (for small $d(x_0,x)$) and  $\mu
\in \; \hbox{RDC}(X)$, we have
$$
 I_1(t) \! \geq \!\!\!\!\!\!\! \int\limits_{A^2a_1 t\leq d(x_{0},x)< A^3 a_1 t}\!\!\!\!\!\!\!\!\! \bigg(
 \frac{v(A^2 a_1 t)}{w(t)}\bigg)^{q(x)}\!\!\!\!\! \big(\mu B(x_0, t)\big)^{(\alpha -1/p_0(x))q(x)} d\mu(x) $$

 $$
  \geq \bigg( \frac{v(A^2 a_1 t)}{w(t)}\bigg)^{q_-} \!\!\!\!\!\! \int\limits_{A^2 a_1 t\leq d(x_{0},x)<A^3 a_1 t}
  \!\!\!\!\!\!\!  \big(\mu B(x_0, t)\big)^{(\alpha- 1/p_0(x))q(x)}d\mu(x) \geq c \bigg( \frac{v(A^2 a_1 t)}{w(t)}
  \bigg)^{q_-}.
$$
Hence, $ \overline{c}:= \overline{\lim\limits_{t\to 0}}
\frac{v(A^2 a_1 t)}{w(t)} <\infty$. Further, if $t>b_2$, where
$b_2$ is a large number, then since $p$ and $q$  are  constants,
for $d(x_0, x)>t$, we have that
\begin{eqnarray*}
 I_1(t) \geq   \bigg( \int\limits_{A^2 a_1 t\leq d(x_{0},x)<A^3 a_1 t} v(d(x_0,x))^{q_c}
 \big(\mu B(x_0, t)\big)^{(\alpha-1)q_c}d\mu(x) \bigg)\\
 \times \bigg( \int\limits_{B(x_0,t)} w^{-(p_c)'}(x) d\mu(x) \bigg)^{ q_c/ (p_c)'} d\mu(x) \\
  \geq C \bigg( \frac{v(A^2 a_1 t)}{w(t)}\bigg)^{q_c} \!\!\!\! \int\limits_{A^2 a_1 t\leq d(x_{0},x)<A^3 a_1 t}
  \!\!\!\! \big(\mu B(x_0, t)\big)^{(\alpha- 1/p_c)q_c}d\mu(x) \geq c \bigg( \frac{v(A^2 a_1t)}{w(t)}\bigg)^{q_c}.
\end{eqnarray*}

In the last inequality we used the fact that $\mu$ satisfies the
reverse doubling condition.

Now we show that the condition $I_1 <\infty$  implies
\begin{eqnarray*}
 \sup_{t>0} I_2(t) := \sup\limits_{t>0}\int \limits_{d(x_{0},x)\leq t} (v(d(x_0,x)))^{q(x)}
 \bigg(\int\limits_{d(x_{0},y)> t}\! \! \!
w^{-({\widetilde{p}}_{1})'(x)} (d(x_0,y))\\
\times \big( \mu(B_{x_0y})
\big)^{(\alpha-1)(\widetilde{p}_1)'(x)}d\mu(y)\bigg)^{\frac{q(x)}{({\widetilde{p}}_{1})'(x)}}d\mu(x)<\infty
\end{eqnarray*}

Due to monotonicity of functions $v$ and $w$, the condition $p\in
LH(X)$, Proposition 1.4, Lemma 1.7, Lemma 1.9  and the assumption
that $p$ has a minimum at $x_0$, we find that for all $t>0$,
\begin{eqnarray*} I_2 (t) \leq \int\limits_{d(x_{0},x)\leq t}\Big(\frac{v(t)}{w(t)}\Big)^{q(x)}
\Big(\mu \big( B(x_0, t)\big)\Big)^{(\alpha-1/p(x_0))q(x)} d\mu(x)\\
 \leq  c \int\limits_{d(x_{0},x)\leq t}\Big(\frac{v(t)}{w(t)}\Big)^{q(x)} \Big(\mu \big( B(x_0, t)
 \big)\Big)^{\big(\alpha-1/p(x_0)\big)q(x_0)} d\mu(x) \\
  \leq  c \bigg(\int\limits_{d(x_{0},x)\leq t}\Big(\frac{v(A^2 a_1 t)}{w(t)}\Big)^{q(x)} d\mu(x)\bigg)
  \Big(\mu \big( B(x_0, t)\big)\Big)^{-1} \leq C.
  \end{eqnarray*}
Now Theorem 3.2 completes the proof. $\Box$

\vskip+0.1cm

%Theorem 3.5
{\bf Theorem 3.5.} {\em  Let $(X,d,\mu)$ be an $SHT$ with
$L<\infty$. Suppose that $p$, $q$ and $\alpha$ are measurable
functions on $X$ satisfying the conditions: $1<p_-\leq p(x)\leq
q(x)\leq q_+<\infty$ and $1/p_-<\alpha_-\leq \alpha_+<1$. Assume
that there is a point $x_0$ such that  $\mu\{ x_0 \}=0$ and $p,q,
\alpha \in \hbox{LH}(X, x_0)$. Suppose also that $w$ is a positive
increasing function on $(0,2L)$.Then the inequality
$$ \|  \big(T_{\alpha(\cdot)} f\big) v\|_{L^{q(\cdot)}(X)} \leq c \| w (d(x_0,\cdot)) f(\cdot) \|_{L^{p(\cdot)}(X)} $$
holds if the following two conditions are satisfied:
\begin{eqnarray*}
\widetilde{I}_1:= \sup\limits_{0<  t\leq L}\int\limits_{t\leq  d(x_{0},x)\leq L}\bigg(\frac{v(x)}{\big(\mu B_{x_0x}\big)^{1-\alpha(x)}}\bigg)^{q(x)}\\
\times \Big(\int\limits_{d(x_{0},x)\leq t}  w^{-(p_{0})'(x)}(d(x_0,y))d\mu(y)\Big)^{\frac{q(x)}{(p_{0})'(x)}}d\mu(x)<\infty;\\
 \widetilde{I}_2:=\sup\limits_{0< t\leq L}\!\!\! \int\limits_{d(x_{0},x)\leq t}\!\!\! \big(v(x)\big)^{q(x)}\bigg(\int\limits_{t\leq  d(x_{0},x)\leq L}
\!\!\! \Big( w(d(x_0,y)) \\
\times \big( \mu B_{x_0y}\big)^{1-\alpha(x)}
\Big)^{-(p_{1})'(x)}d\mu(y)
\bigg)^{\frac{q(x)}{(p_1)'(x)}}d\mu(x)<\infty.
\end{eqnarray*}}

\vskip+0.1cm

{\em Proof.} For simplicity assume that $L=1$. First observe that
by Lemma 1.9 we have  $p,q,\alpha\in {\mathcal{P}}(1)$. Suppose
that $f\geq 0$ and $S_{p} \big( w(d(x_0,\cdot)) f(\cdot) \big)
\leq 1$. We will show that $S_{q} \big( v (T_{\alpha(\cdot)}
f)\big) \leq C$.

We have
$$S_{q} \big( v T_{\alpha(\cdot)} f \big)   \leq C_q \bigg[ \int\limits_{X}  \bigg( v(x) \int\limits_{ d(x_0,y)\leq
d(x_0,x)/(2a_1) } f(y) \big( \mu B_{xy} \big)^{ \alpha(x)-1 } d\mu(y) \bigg)^{q(x)} d \mu(x) $$
$$
+  \int_X  \bigg(v(x) \int\limits_{d(x_0,x)/(2a_1) \leq d(x_0,y)
\leq  2a_1 d(x_0,x)} f(y) \big( \mu B_{xy}
\big)^{\alpha(x)-1}d\mu(y) \bigg)^{q(x)} d\mu(x) $$
$$
 + \int\limits_X  \bigg(v(x) \int\limits_{d(x_0,y) \geq   2a_1 d(x_0,x)} f(y) \big( \mu B_{xy}
 \big)^{\alpha(x)-1}d\mu(y) \bigg)^{q(x)} d\mu(x) \bigg] := C_q [I_1 +I_2+I_3].
$$

First observe that by virtue of  the doubling condition for $\mu$,
Remark 1.1 and simple calculation we find that $ \mu \big(
B_{x_0x}\big) \leq c \mu \big(B_{xy}\big)$. Taking into account
this estimate  and Theorem 2.1 we have that

$$ I_1 \leq c  \int_X  \bigg( \frac{v(x)}{ \big( \mu B_{x_0x} \big)^{1-\alpha(x)}} \int\limits_{d(x_0,y)<
d(x_0,x)} f(y) d\mu(y) \bigg)^{q(x)} d\mu(x) \leq C. $$

Further, it is easy to see that if $d(x_0,y)\geq 2a_1 d(x_0,x)$,
then the triangle inequality for $d$ and the doubling condition
for $\mu$ yield that $\mu B_{x_0y} \leq c \mu B_{xy}.$  Hence due
to Proposition 1.5 we see that $ \big( \mu
B_{x_0y}\big)^{\alpha(x)-1} \geq c  \big( \mu
B_{xy}\big)^{\alpha(y)-1}$ for such $x$ and $y$. Therefore,
Theorem 2.2 implies that $ I_3 \leq C.$

It remains to estimate $I_2$. Let us denote:
$$ E^{(1)}(x) \!:=\! \overline{B}_{x_0 x}\setminus B\big(x_0, d(x_0,x)/(2a_1)\big); \;\; E^{(2)}(x):=
\overline{B}\big(x_0, 2a_1 d(x_0,x)\big) \setminus B_{x_0  x}.$$
Then we have that
$$
I_2 \leq C \bigg[ \int\limits_{X} \Big[v(x)
\int\limits_{E^{(1)}(x)} f(y) \big( \mu B_{x y}\big)^{\alpha(x)-1}
d\mu (y) \Big]^{q(x)} d\mu(x) $$
$$
 + \int\limits_{X} \Big[v(x) \int\limits_{E^{(2)}(x)} f(y) \big( \mu B_{x y}\big)^{\alpha(x)-1} d\mu (y)
 \Big]^{q(x)} d\mu(x)\bigg] := c[ I_{21}+I_{22}].
$$
Using H\"older's inequality for the classical Lebesgue spaces we
find that
$$I_{21} \leq   \int\limits_X  v^{q(x)} (x) \bigg( \int\limits_{E^{(1)}(x)} w^{p_0(x)} (d(x_0,y))(f(y))^{p_0(x)}
d\mu(y) \bigg)^{q(x)/p_0(x)}$$
$$ \times \bigg( \int\limits_{E^{(1)}(x)}
w^{-(p_0)'(x)}(d(x_0,y))\big( \mu
B_{xy}\big)^{(\alpha(x)-1)(p_0)'(x)} d\mu(y)
\bigg)^{q(x)/(p_0)'(x)} d\mu(x).
$$

Denote the first inner integral by $J^{(1)}$ and the second one by
$J^{(2)}$.

By using the fact that $p_0(x) \leq p(y)$, where $y\in
E^{(1)}(x)$, we see that $J^{(1)} \leq \mu( B_{x_0 x} )+
\int\limits_{E^{(1)}(x)} (f(y))^{p(y)} \big(
w(d(x_0,y))\big)^{p(y)} d\mu(y) $, while by applying  Lemma 1.7,
for $J^{(2)}$,  we have that
\begin{eqnarray*}J^{(2)} \leq c w^{-(p_0)'(x)} \Big( \frac{d(x_0, x)}{2a_1}\Big) \int\limits_{E^{(1)}(x)} \Big( \mu B_{x y} \Big)^{(\alpha(x)-1)(p_0)'(x)} d\mu(y)  \\
\leq c w^{-(p_0)'(x)}\Big(\frac{d(x_0,x)}{2a_1}\Big)  \Big( \mu
B_{x_0 x} \Big)^{(\alpha(x)-1)(p_0)'(x)+1}.
\end{eqnarray*}

Summarizing these estimates for $J^{(1)}$ and $J^{(2)}$ we
conclude that
\begin{eqnarray*}
I_{21} \leq  \int\limits_X  v^{q(x)} (x) \big( \mu B_{x_0x}\big)^{q(x)\alpha(x)} w^{-q(x)}\Big(\frac{d(x_0,x)}{2a_1}\Big) d\mu(x)  + \int\limits_X  v^{q(x)} (x) \\
\times \bigg( \int\limits_{E^{(1)}(x)} w^{p(y)}
(d(x_0,y))(f(y))^{p(y)} d\mu(y) \bigg)^{q(x)/p_0(x)} \big( \mu
B_{x_0x}\big)^{q(x)(\alpha(x)-1/p_0(x))} \\
\times w^{-q(x)}\Big(\frac{d(x_0,x)}{2a_1}\Big)d\mu(x) =:
I_{21}^{(1)}+ I_{21}^{(2)}.
\end{eqnarray*}

By applying monotonicity of $w$, the reverse doubling property for
$\mu$ with the  constants  $A$ and $B$ (see Remark 1.3), and  the
condition $\widetilde{I}_1<\infty$   we have that
$$I_{21}^{(1)} \!\!\leq c \!\!\sum_{k=-\infty}^{0} \int\limits_{ \overline{B} (x_0, A^k)
\setminus B(x_0, A^{k-1})}\!\!\!\!\!\!\!\!\!\!\!\!\!\!\!\!v(x)^{q(x)} \bigg( \!\!\!\!\!\!\!
\int\limits_{B \big( x_0, \frac{A^{k-1}}{2a_1}\big)}\!\!\!\!\!\!\!\!\!\! w^{-(p_0)'(x)} (d(x_0,y)) d\mu(y)
\bigg)^{\frac{q(x)}{(p_0)'(x)}} $$

$$
\times \big( \mu
B_{x_0,x}\big)^{\frac{q(x)}{p_0(x)}+(\alpha(x)-1)q(x)}d\mu(x)
\leq c \sum_{k=-\infty}^0 \!\!\! \Big( \mu \overline{B} (x_0,
A^k)\Big)^{q_-/p_+} $$

$$
\times \int\limits_{\overline{B}(x_0, A^k) \setminus B(x_0,
A^{k-1})} \!\!\!\!\!\!\!\!\! v(x)^{q(x)} \bigg(\int\limits_{B
\big( x_0, A^{k}\big)} \!\!\!\!\!\!\!\! w^{-(p_0)'(x)} (d(x_0,y))
d\mu(y) \bigg)^{\frac{q(x)}{(p_0)'(x)}}
$$

$$
 \times \big( \mu B_{x_0,x}\big)^{q(x)(\alpha(x)-1)}d\mu(x)
 \leq c \sum_{k=-\infty}^0 \Big( \mu \bar{B}(x_0, A^k) \setminus B(x_0, A^{k-1}) \Big)^{q_-/p_+} $$

 $$
 \leq c \sum_{k=-\infty}^0
\int\limits_{\mu \bar{B}(x_0, A^k) \setminus B(x_0, A^{k-1})}
\!\!\!\!\!\!\!\! \big( \mu B_{x_0,x}\big)^{q_-/p_+-1} d\mu(y)
$$

$$
 \leq c \int_{X} \big( \mu B_{x_0,x}\big)^{q_-/p_+-1} d\mu(y)
<\infty.
$$

 Due to the facts that $q(x) \geq p_0(x)$,
$S_p\big(w\big(d(x_0,\cdot)f(\cdot)\big)\big)\leq 1$,
$\widetilde{I}_1<\infty$  and $w$ is increasing, for
$I_{21}^{(2)}$, we find that
 \begin{eqnarray*}
 I_{21}^{(2)} \leq c \sum_{k=-\infty}^0 \bigg( \int\limits_{\mu \bar{B}(x_0, A^{k+1}a_1) \setminus B(x_0, A^{k-2})} w^{p(y)}(d(x_0,y))(f(y))^{p(y)} d\mu(y) \bigg) \\ \times\bigg( \int\limits_{\mu \bar{B}(x_0, A^k) \setminus B(x_0, A^{k-1})}v^{q(x)}(x) \bigg( \int\limits_{B(x_0,A^{k-1})} w^{-(p_0)'(x)} (d(x_0,y)) d\mu(y) \bigg)^{\frac{q(x)}{(p_0)'(x)}} \\
  \times\big(\mu B_{x_0,x}\big)^{(\alpha(x)-1)q(x)}d\mu(x) \bigg) \leq c S_p(f(\cdot)w(d(x_0,\cdot)) \leq c.
  \end{eqnarray*}

Analogously, it follows the estimate for $I_{22}$. In this case we
use the condition $\widetilde{I}_2<\infty$ and the fact that
$p_1(x) \leq p(y)$ when $d(x_0,y) \leq d(x_0,y) < 2a_1 d(x_0,x)$.
The details are omitted. The theorem is proved.  $\Box$
\vskip+0.1cm

Taking into account the proof of Theorem 3.5 we can easily derive
the following statement proof of which is omitted: \vskip+0.1cm

%Theorem 3.6
{\bf Theorem 3.6.} {\em Let $(X,d,\mu)$ be an $SHT$ with
$L<\infty$. Suppose that $p$, $q$ and $\alpha$ are measurable
functions on $X$ satisfying the conditions $1<p_-\leq p(x)\leq
q(x)\leq q_+<\infty$ and $1/p_-<\alpha_-\leq \alpha_+<1$. Assume
that there is a point $x_0$ such that  $p,q, \alpha \in
\hbox{LH}(X, x_0)$ and $p$ has a minimum at $x_0$. Let $v$ and $w$
be  positive increasing function on $(0,2L)$ satisfying  the
condition $J_1<\infty$ $($ see Theorem  $3.4\;)$. Then inequality
$(7)$ is fulfilled.}

\vskip+0.1cm

%Theorem 3.7
{\bf Theorem 3.7.} {\em  Let  $(X, d, \mu)$ be an SHT with
$L<\infty$ and let $\mu$ be upper Ahlfors $1$-regular. Suppose
that  $1<p_-\leq p_+<\infty$ and that $p \in {\overline{LH}}(X)$.
Let $p$ have a minimum at $x_0$. Assume that $\alpha$ is constant
satisfying the condition $\alpha< 1/p_+$. We set $q(x)=
\frac{p(x)}{1-\alpha p(x)}$.  If $v$ and $w$ are positive
increasing functions on $(0,2L)$ satisfying the condition

$$E:=\!\!\!\sup\limits_{0\leq t\leq L}\!\!\!\!\!\!\! \int\limits_{t< d(x_{0},x)\leq L} \!\!\!\!\!\!\!\!\bigg(\frac{v(d(x_0,x))}{\big(d(x_0,x)\big)^{1-\alpha}}\bigg)^{q(x)}  \bigg(\!\!\! \int\limits_{d(x_{0},x)\leq t} \!\!\!\!\!\!\!\!w^{-(p_{0})'(d(x_0,x))}(y)d\mu(y)\bigg)^{\frac{q(x)}{(p_{0})'(x)}}\!\!\! d\mu(x)<\infty, $$
then the inequality

$$ \| v \big(d(x_0,\cdot)\big) (I_{\alpha}f)(\cdot) \|_{L^{q(\cdot)}(X)} \leq c \| w\big(d(x_0,\cdot)\big) f(\cdot)\|_{L^{p(\cdot)}(X)} $$
holds.}

\vskip+0.1cm

{\em Proof} is similar to that of Theorem 3.4. We only discuss
some details. First observe that due to Remark 1.2 we have that
$p\in {\mathcal{P}}(N)$, where $N= a_1(1+2a_0)$. It is easy to
check that the condition $E<\infty$ implies that
$\frac{v(A^2a_1t)}{w(t)} \leq C $ for all t, where the constant
$A$ is defined in Definition 1.10 and $a_1$ is from the triangle
inequality for $d$. Further, Lemmas 1.7, 1.9, the fact that $p$
has a minimum at $x_0$ and the inequality
$$\int\limits_{d(x_0,y)>t}\big( d(x_0,y)\big)^{(\alpha-1)(p_1)'(x)} d\mu(y) \leq c t^{(\alpha-1)(p_1)'(x)+1},$$
where the constant $c$ does not depend on $t$ and $x$, yield that
$$ \sup\limits_{0\leq t \leq L}\!\!\int\limits_{d(x_{0},x)\leq t} \!\!\! (v(d(x_0,x)))^{q(x)}  \bigg(\int\limits_{d(x_{0},y)> t}\! \! \!\!\!
\bigg(\frac{w (d(x_0,y))}{\big( d(x_0,y)
\big)^{1-\alpha}}\bigg)^{-(p_{1})'(x)} \!\!\!\!
d\mu(y)\bigg)^{\frac{q(x)}{(p_{1})'(x)}}\!\!d\mu(x)<\infty.
$$
Theorem 3.3 completes the proof. $\!\!\!\Box$ \vskip+0.2cm

%Example 3.7
{\bf Example 3.8.} {\em  Let $v(t)=t^{\gamma}$ and $w(t)=
t^{\beta}$, where $\gamma$ and $\beta$ are constants satisfying
the condition $ 0\leq \beta< 1/(p_-)'$, $\gamma\geq \max\{0,\;\;
1-\alpha-\frac{1}{q_+}-\frac{q_-}{q_+}(-\beta+\frac{1}{(p_-)'})\}$.
Then $(v,w)$ satisfies  the conditions of Theorem $3.4$.}

\section{Maximal and Singular Operators}
Let
$$ Kf (x)= p.v. \int\limits_X k(x,y)f(y) d\mu(y), $$
where  $k: X\times X\setminus \{(x,x): x\in X\} \to {\Bbb{R}}$ be a measurable function satisfying the conditions:
\begin{gather*}
|k(x,y)|\leq \frac{c}{\mu B(x, d(x,y))}, \;\; x,y\in X, \;\; x\neq y; \\
 |k(x_1,y)-k(x_2,y)|+ |k(y, x_1)-k(y, x_2)| \leq c \omega \Big( \frac{d(x_2, x_1)}{d(x_2, y)}\Big) \frac{1}{\mu B(x_2, d(x_2,y))}
 \end{gather*}
for all $x_1, x_2$ and $y$ with $d(x_2,y)\geq c d(x, x_2)$, where
$\omega$ is a positive non-decreasing function on $(0,\infty)$
which satisfies the  $\Delta_2$ condition: $\omega(2t)\leq c
\omega(t)$ ($t>0$);  and the Dini condition: $\int_0^1
\big(\omega(t)/t\big) dt <\infty$.

We also assume that for some constant $s$, $1<s<\infty$, and all
$f\in L^{s}(X)$ the limit $Kf(x)$ exists almost everywhere on $X$
and that $K$ is bounded in $L^{s}(X)$. \vskip+0.1cm It is known
(see, e.g., \cite{EdKoMe}, Ch. 7) that if  $r$ is constant such
that $1<r<\infty$, $(X, d, \mu)$ is an SHT and  the weight
function $w\in A_r(X)$, i.e.
$$ \sup_{B} \bigg(\frac{1}{\mu (B)} \int\limits_B w(x) d\mu(x)\bigg)  \bigg( \frac{1}{\mu (B)}
\int\limits_B w^{1-r'}(x) d\mu(x)\bigg)^{r-1}<\infty,$$
where the supremum is taken over all balls $B$ in $X$, then
the one--weight inequality $ \| w^{1/r} Kf\|_{L^{r}(X)} \leq c \|w^{1/r} f \|_{L^{r}(X)}$ holds.
\vskip+0.1cm

The boundedness of Calder\'on--Zygmund operators in $L^{p(\cdot)}({\Bbb{R}}^n)$ was establish in \cite{DiRu}.

\vskip+0.1cm

{\bf Theorem E \cite{KoSa8}.} {\em Let $(X, d,\mu)$ be an \hbox{SHT}.  Suppose that $p\in {\mathcal{P}}(1)$.  Then the singular operator $K$ is bounded in $L^{p(\cdot)}(X)$.}
\vskip+0.1cm

Before formulating the main results of this section we introduce the notation:
$$ \overline{v}(x):= \frac{v(x)}{\mu (B_{x_0 x})},\;\;\; \widetilde{w}(x) := \frac{1}{w(x)}, \;\; \widetilde{w}_1(x):=
\frac{1}{w(x)\mu (B_{x_0  x})}.$$

The following  statements follows in the same way as Theorem 3.1 was proved. In this case Theorem 1.2 (for the maximal operator) and Theorem E (for singular integrals)  are used instead of Theorem C.  Details are omitted.
%Theorem 4.1
\begin{theorem}Let $(X, d, \mu)$ be an $\hbox{SHT}$ and let  $1<p_-\leq
p_+<\infty$. Further suppose that $p\in {\mathcal{P}}(1)$. If
$L=\infty$, then we assume that $p$ is constant outside some ball $B(x_0,a)$. Then the
inequality

$$ \| v (Nf) \|_{L^{p(\cdot)}(X)} \leq C \| w f \|_{L^{p(\cdot)}(X)}, \eqno{(8)}$$
where $N$ is $M$ or $K$, holds if the following three conditions are satisfied:

$(a)$ $\;\;\;\; T_{\overline{v}, \widetilde{w}}$ is bounded in $L^{p(\cdot)}(X)$;

$(b)$ $\;\;\;\; T'_{v, \widetilde{w}_1}$ is bounded in
$L^{p(\cdot)}(X)$;

$(c)$  there is a positive constant $b$ such that one of the
following two conditions hold:   $1)\;\; v_+(F_x) \leq b w(x)$ $\mu-$ a.e. $x\in X$; $2)\;\; v(x)\leq b \; w_-(F_x)$ $\mu-$ a.e. $x\in X $,
where $F_x$ is the set depended on $x$ which is defined in Section $3$.
\end{theorem}

The next two statements are direct consequences of Theorems 4.1, 2.1 and  2.2.

%Theorem 4.2
\begin{theorem} Let $(X, d, \mu)$ be an $\hbox{SHT}$ and let
$1<p_-\leq p_+<\infty$. Further suppose that $p\in
{\mathcal{P}}(1)$. If $L=\infty$, then we assume that  $p\equiv$
$p_c\equiv$ const  outside some ball $B(x_0,a)$.  Let $N$ be $M$ or $K$. Then
inequality $(8)$ holds if:

\begin{eqnarray*}
(i) \;\;   \sup\limits_{0< t\leq L}\int\limits_{t< d(x_{0},x) \leq L}\bigg(\frac{v(x)}{\mu B_{x_0x}}\bigg)^{p(x)} \bigg(\int\limits_{\overline{B}(x_0,t)}w^{-({\widetilde{p}}_{0})'(x)}(y)d\mu(y)\bigg)^{\frac{p(x)}{({\widetilde{p}}_{0})'(x)}}d\mu(x)<\infty;\\
(ii)\;\;     \sup\limits_{0< t\leq L}\int\limits_{\overline{B}(x_0, t)} \big(v(x)\big)^{p(x)} \bigg(\int\limits_{t< d(x_{0},x)\leq L}
\bigg(\frac{w(y)}{\mu B_{x_0y}}\bigg)^{-({\widetilde{p}}_{1})'(x)}d\mu(y)\bigg)^{\frac{p(x)}{({\widetilde{p}}_1)'(x)}}d\mu(x)<\infty;
\end{eqnarray*}
$(iii)\;\;\;\; $  condition $(c)$ of the previous theorem is satisfied.\end{theorem}

{\em Remark} 4.1. It is known  (see \cite{EdKo}) that if $p\equiv$ const, then conditions (i) and (ii)
(written for $X={\Bbb{R}}$, the Euclidean distance and the Lebesgue measure) of Theorem  4.2 are also necessary
for the two--weight inequality
$$ \| v (Hf) \|_{L^{p(\cdot)}({\mathbb{R}})} \leq C \| w f \|_{L^{p(\cdot)}({\mathbb{R}})}, $$
where $H$ is the Hilbert transform on ${\mathbb{R}}$: $(Hf)(x)=\;
\text{p.v.}\; \int\limits_{{\mathbb{R}}} \frac{f(t)}{x-t}dt$.
\vskip+0.1cm {\em Remark} 4.2. If $p\equiv$ const and $N=M$, then
condition (i) of Theorem 4.2 is necessary for (8). This follows
from the obvious estimate $Mf(x) \geq \frac{c}{\mu ( B_{x_0x} )}
\int\limits_{B_{x_0x}} f(y) d\mu(y)$ ($f\geq 0$)  and Remark 2.2.

%Theorem 4.3
\begin{theorem}
Let $(X, d, \mu)$ be an $\hbox{SHT}$ without atoms.  Let
$1<p_-\leq p_+<\infty$. Assume  that $p$ has a minimum at $x_0$
and that $p \in \hbox{LH}(X)$. If $L=\infty$ we also assume  that
$p\equiv p_c \equiv$ const outside some ball $B(x_0,a)$. Let $v$
and $w$ be positive increasing functions on $(0,2L)$. Then the
inequality
$$ \| v(d(x_0,\cdot)) (N f) (\cdot) \|_{L^{p(\cdot)}(X)} \leq c \| w(d(x_0,\cdot))f(\cdot)\|_{L^{p(\cdot)}(X)}, \eqno{(9)} $$
 where $N$ is $M$ or $K$,       holds if the following  condition is satisfied:

$$  \sup\limits_{0<t \leq  L}\int\limits_{t< d(x_{0},x) \leq  L}\!\!\!\!\!\!\!\!\!\!\!\!\bigg(\frac{v(d(x_0,x))}{\mu(B_{x_0x}) } \bigg)^{p(x)} \bigg(\int\limits_{\overline{B}(x_0, t)} \!\!\!\!\! w^{-({\widetilde{p}}_{0})'(x)}(d(x_0,y)) d\mu(y)\bigg)^{\frac{p(x)}{({\widetilde{p}}_{0})'(x)}}d\mu(x)<\infty. $$
\end{theorem}

Proof of this statement is similar to that of Theorem 3.4;
therefore we omit it. Notice that Lemma 1.9 yields  that $p\in
LH(X)\Rightarrow p\in {\mathcal{P}}(1).$

\begin{example}
 Let $(X, d, \mu)$ be a quasimetric measure space with  $L<\infty$. Suppose that $1<p_-\leq p_+<\infty$ and  $p\in LH(X)$.  Assume that the measure $\mu$ is upper and lower Ahlfors $1-$ regular. Let there exist $x_0\in X$ such that $p$ has a minimum at $x_0$. Then the condition
$$ S:= \sup\limits_{0<t \leq  L}\int\limits_{t< d(x_{0},x)\leq L}\!\!\!\!\!\!\!\!\!\!\!\!\bigg(\frac{v(d(x_0,x))}{\mu(B_{x_0x}) } \bigg)^{p(x)} \bigg(\int\limits_{\overline{B}(x_0, t)} \!\!\!\!\!w^{-p'(x_0)}(d(x_0,y))d\mu(y)\bigg)^{\frac{p(x)}{p'(x_0)}}d\mu(x)<\infty $$
is satisfied for the weight functions
$ v(t)= t^{1/p'(x_0)}$, $w(t)=t^{1/p'(x_0)}\ln \frac{2L}{t}$ and, consequently, by Theorem $4.3$ inequality $(9)$ holds, where $N$ is $M$ or $K$.
\end{example}

Indeed, first observe that $v$ and $w$ are both increasing on $[0,L]$. Further, it is easy to check that  $S \leq c \sup\limits_{0< t \leq L}  V(t)  \Big(W(t)\Big)^{\frac{p(x_0)}{p'(x_0)}}<\infty$ because $ W(L) <\infty$.

By using the representation formula of a general integral by improper integral and the fact that $\mu$ is Ahlfors $1-$ regular, it follows that
$ W(t) \leq C_1 \ln^{-1} \frac{2L}{t}$ and $ V(t) \leq C_2 \ln \frac{2L}{t}$ for  $0<t\leq L$,
where the positive constants does not depend on $t$. Hence the result follows.

Observe that for the constant $p$ both weights $v$ and $w$ are outside the Muckenhoupt class $A_p(X)$ (see e.g. \cite{EdKoMe}, Ch. 8).

\vskip+0.1cm

%For constant $p$, the example of this  weight pair in Euclidean spaces was first given in \cite{EdKo}.

{\bf Acknowledgement.}  The first and second authors were
partially supported by the Georgian National Science Foundation
Grant (project numbers: No. GNSF/ST09/23/3-100 and No.
GNSF/ST07/3-169). The part of this work was fulfilled in Abdus
Salam School of Mathematical sciences, GC University, Lahore. The
second and third authors are grateful to the Higher Educational
Commission of Pakistan for the financial support.
%09 23 3-100
%
%
%
%\bibliographystyle{gITR}

Authors' Addresses

V. Kokilashvili: A. Razmadze Mathematical Institute, 1. M. Aleksidze Str., 0193 Tbilisi, Georgia and Faculty
of Exact and Natural Sciences, I. Javakhishvili Tbilisi State University 2,
University St., Tbilisi 0143 Georgia \\ e-mail: kokil@rmi.acnet.ge.\\

A. Meskhi: A. Razmadze Mathematical Institute, 1. M. Aleksidze Str., 0193 Tbilisi, Georgia and Department
of Mathematics,  Faculty of Informatics and Control Systems, Georgian Technical University, 77, Kostava St.,
Tbilisi, Georgia.\\e-mail: meskhi@rmi.acnet.ge\\

M. Sarwar: Abdus Salam School of Mathematical Sciences, GC University, 68-B New Muslim Town, Lahore,
Pakistan \\e-mail: sarwarswati@gmail.com

\end{document}